\documentclass[a4paper,11pt]{amsart}

\usepackage{amssymb, amsthm, amsmath,amsfonts}
\usepackage{tikz}

\usepackage{mathtools}

\usepackage{bbm}

\usepackage[hmargin=3.5cm,foot=0.7cm]{geometry}
\usepackage[utf8x]{inputenc}

\usepackage[english]{babel}
\usepackage[abbrev]{amsrefs}

\usepackage{tikz}
\usepackage{tikz-cd}
\usetikzlibrary{matrix,arrows}
\usepackage{pgfplots}
\pgfplotsset{width=7cm,compat=1.8}
\usepgfplotslibrary{polar}
\usepackage{pifont}

\usepackage{dsfont}
\usepackage[]{hyperref}
\usepackage{cleveref}

\usepackage{enumitem}
\usepackage{verbatim}

\usetikzlibrary{decorations.pathreplacing}

\usepackage{xcolor}
\usepackage[textsize=tiny]{todonotes}

\newtheorem{theorem}{Theorem}[section]
\newtheorem*{theorem*}{Theorem}
\newtheorem{definition}[theorem]{Definition}

\theoremstyle{plain}
\newtheorem{corollary}[theorem]{Corollary}
\newtheorem*{corollary*}{Corollary}
\newtheorem{lemma}[theorem]{Lemma}
\newtheorem{proposition}[theorem]{Proposition}

\newcounter{mt}

\newtheorem{MainTheorem}[mt]{Theorem}

\newcounter{pt}
\newtheorem{Problem}[pt]{Problem}

\theoremstyle{definition}
\newtheorem{example}[theorem]{Example}
\newtheorem{remark}[theorem]{Remark}

\makeatletter
\newcommand{\tpitchfork}{%
  \vbox{
    \baselineskip\z@skip
    \lineskip-.52ex
    \lineskiplimit\maxdimen
    \m@th
    \ialign{##\crcr\hidewidth\smash{$-$}\hidewidth\crcr$\pitchfork$\crcr}
  }%
}
\makeatother

\BibSpec{article}{%
  +{}{\PrintAuthors}  		{author}
  +{,}{ \textit}     		{title}
  +{,}{ }             		{journal}
  +{}{ \textbf}       		{volume}
  +{}{ \parenthesize} 		{date}
  +{}{, no. } 		{number}
  +{,}{ }      	      		{conference}
  +{,}{ }      	      		{book}
  +{,}{ }            		{pages}
  +{,}{ }            	 	{note}
  +{,}{ }            	 	{status}
  +{,}{  \texttt } {eprint}
  +{.}{}              {transition}
}

\BibSpec{book}{%
  +{}{\PrintAuthors}  {author}
  +{,}{ \textit}      {title}
  +{,}{ }             {publisher}
  +{,}{ }             {place}
  +{,}{ }             {date}
  +{.}{}              {transition}
}

\BibSpec{incollection}{%
    +{}  {\PrintAuthors}                {author}
    +{,} { \textit}                     {title}
    +{.} { }                            {part}
    +{:} { \textit}                     {subtitle}
    +{,} { \PrintContributions}         {contribution}
    +{,} { \PrintConference}            {conference}
    +{,} { }                            {booktitle}
    +{,} { pp.~}                        {pages}
    +{,}{ }       		{series}
    +{}{ \textbf}       		{volume}
    +{,} { }                            {publisher}
    +{,} { }                 {date}
    +{,} { }                            {status}
    +{,} { \PrintDOI}                   {doi}
    +{,} { available at \eprint}        {eprint}
    +{}  { \parenthesize}               {language}
    +{}  { \PrintTranslation}           {translation}
    +{;} { \PrintReprint}               {reprint}
    +{.} { }                            {note}
    +{.} {}                             {transition}
}

\newcommand{\CC}{\mathbb{C}}

\newcommand{\NN}{\mathbb{N}}

\newcommand{\RR}{\mathbb{R}}

\DeclareMathOperator{\OO}{O}
\DeclareMathOperator{\UU}{U}

\DeclareMathOperator{\interior}{int}

\DeclareMathOperator{\cl}{cl}
\DeclareMathOperator{\image}{im}

\newcommand{\DD}{\mathbb{D}}

\newcommand{\JOp}[2]{\mathrm{J}_{#1,#2}}

\newcommand{\CIntBody}[2]{\mathrm{I}_{#1,#2}}
\newcommand{\IntBody}{\mathrm{I}}

\newcommand{\convexbodies}{\mathcal{K}}
\newcommand{\starbodiesO}{\mathcal{S}_0}

\newcommand{\unitsurf}[1]{\mathbb{S}^{#1}}

\newcommand{\TT}[1]{\mathrm{T}_{#1}}

\newcommand{\CompProj}{\Pi}
\newcommand{\MixedVol}[1]{\mathrm{V}(#1)}
\newcommand{\Vol}[2]{\mathrm{V}_{#1}(#2)}
\newcommand{\voln}{\mathrm{V}}

\newcommand{\Fourier}{\mathrm{F}}
\newcommand{\DualMixedVol}[2]{\tilde{\mathrm{V}}_{#1}(#2)}
\newcommand{\GAMMA}[1]{\Gamma\left(#1\right)}
\newcommand{\sphericalFourier}[1]{{{\mathbf{F}}}_{#1}}
\newcommand{\innerProduct}[1]{\left\langle #1 \right\rangle}
\newcommand{\surfArea}[1]{S_{#1}}

\newcommand{\imgSetPhi}[2]{#1({#2})}

\title[]{Busemann--Petty type problems\\ on complex vector spaces}
\author{Simon Ellmeyer}
\address{Institute f. Discrete Mathematics and Geometry, TU Wien, 1040 Vienna, Austria}
\email{simon.ellmeyer@tuwien.ac.at}

\author{Georg C. Hofst\"atter}
\address{Institute f. Discrete Mathematics and Geometry, TU Wien, 1040 Vienna, Austria}
\email{georg.hofstaetter@tuwien.ac.at}

\thanks{{\it MSC classification}:
	52A38, 
	51M15, 
	52A40. 
	\\\indent GH was supported by the Austrian Science Fund (FWF), \href{https://doi.org/10.55776/P34446}{doi:10.55776/P34446}.
}
\date{\today}
\begin{document}

\begin{abstract}
Busemann--Petty type problems for the recently introduced complex projection, centroid and $L_p$-intersection body operators are examined. Moreover, it is shown that, as their real counterparts, they can be linked to questions about the spherical Fourier transform.
\end{abstract}

\maketitle

\section{Introduction}
The \emph{Busemann--Petty problem} \cite{Busemann1956} was one of the most famous problems in convex geometry of the twentieth century: Suppose that $K, L \in \convexbodies(\RR^n)$, the space of convex bodies (compact and convex sets in $\RR^n$), are origin-symmetric with non-empty interiors, and assume that
\begin{align}\label{eq:BPproblem}
	\voln_{n-1}(K\cap u^\perp)\leq \voln_{n-1}(L\cap u^\perp), \quad \forall u \in \unitsurf{n-1},
\end{align}
where $\voln_{i}$ denotes the $i$-dimensional volume, $\unitsurf{n-1}$ the unit sphere and $u^\perp$ the hyperplane perpendicular to $u$. Does this imply that $\voln_{n}(K) \leq \voln_n(L)$? The answer is affirmative, if $n \leq 4$, and negative, if $n \geq 5$. Many authors contributed to this solution and worked on related problems over a period of $50$ years (see \cites{Ball1988,Bourgain1991,Giannopoulos1990,Gardner1994b,Gardner1994c, Gardner1999,Koldobsky1997, Larman1975,Lutwak1988,Papadimitrakis1992,Zhang1994,Zhang1994b,Zhang1999, Koldobsky2011b, Koldobsky2000, Koldobsky2003, Rubin2010, Rubin2009, Rubin2004, Yaskin2006c, Yaskin2006b,Zvavitch2005}). A unified proof was given by Gardner, Koldobsky and Schlumprecht in \cite{Gardner1999}.

The dual question to the Busemann--Petty problem is known as \emph{Shephard's problem}~\cite{Shephard1964}:  Suppose that $K, L \in \convexbodies(\RR^n)$ are origin-symmetric with non-empty interiors, and assume that
\begin{align}\label{eq:Shproblem}
	\voln_{n-1}(K\vert u^\perp)\leq \voln_{n-1}(L\vert u^\perp), \quad \forall u \in \unitsurf{n-1},
\end{align}
where $K \vert u^\perp$ denotes the orthogonal projection of $K$ onto $u^\perp$. Does this imply that $\voln_{n}(K) \leq \voln_n(L)$? Here, the complete answer was given independently by Petty~\cite{Petty1967} and Schneider~\cite{Schneider1967}. It is affirmative, if $n \leq 2$, and negative, if $n \geq 3$.

The key to the solution of both problems was to link them to certain geometric operators, namely to the \emph{intersection body} operator for the Busemann--Petty problem (as was first noticed by Lutwak~\cite{Lutwak1988}), and to the \emph{projection body} operator for Shephard's problem. Indeed, the inequality for the volume holds, that is, the answer is affirmative, if, in \eqref{eq:BPproblem}, $K$ is an intersection body, or if, in \eqref{eq:Shproblem}, $L$ is a projection body. Moreover, if $K\in \convexbodies(\RR^n)$ is origin-symmetric, sufficiently regular and \emph{not} a projection body, then there exists an origin-symmetric $L \in \convexbodies(\RR^n)$ such that \eqref{eq:Shproblem} holds, but the inequality for the volume is reversed (\cite{Schneider1967}). A similar statement holds for intersection bodies and \eqref{eq:BPproblem} (\cite{Lutwak1988}). As a consequence, both problems were solved by an analysis of the classes of intersection resp.\ projection bodies.

Interestingly, both problems can also be stated in terms of intersection resp. projection bodies. To see this, denote by $\starbodiesO(\RR^n)$ the set of all star bodies, that is, subsets $K$ of $\RR^n$ that are star-shaped with respect to the origin and have a continuous positive radial function $\rho_K(u)=\sup\{\lambda \geq 0:\lambda u \in K\}$, $u\in \unitsurf{n-1}$. For $K \in \starbodiesO(\RR^n)$, its intersection body $\IntBody K \in \starbodiesO(\RR^n)$ is defined by (\cite{Lutwak1988})
\begin{align*}
	\rho_{\IntBody K}(u)=\Vol{n-1}{K\cap u^\perp}, \quad u \in \unitsurf{n-1}.
\end{align*}
The projection body $\Pi K$ of a convex body $K \in \convexbodies(\RR^n)$ is defined by (see, e.g., \cite{Gardner2006})
\begin{align*}
	h_{\Pi K}(u)=\Vol{n-1}{K\vert u^\perp}, \quad u \in \unitsurf{n-1},
\end{align*}
where $h_K(u)=\sup\{\innerProduct{x,u}:x \in K\}$, $u \in \unitsurf{n-1}$, is the support function of $K$, and $\langle\cdot,\cdot\rangle$ is the Euclidean inner product.

In this way, the Busemann--Petty problem and Shephard's problem are special cases of the following more general class of Busemann--Petty type problems which we vaguely formulate for geometric operators on (subsets of) convex bodies with non-empty interiors and make more precise later.
\begin{Problem}\label{prob:BPgeneralPhi}
Let $\Phi$ be a geometric operator and let $K,L \in \convexbodies(\Phi)$.
	\begin{align*}
		\text{Does} \quad \quad \Phi K \subseteq \Phi L \quad \Longrightarrow \quad \Vol{n}{K}\leq\Vol{n}{L} \quad \quad \text{hold}?
	\end{align*}
\end{Problem}
Here, we denote by $\convexbodies(\Phi) \subset \convexbodies(\RR^n)$ a subset of convex bodies with non-empty interior, for which we will consider Problem~\ref{prob:BPgeneralPhi}. Usually, $\convexbodies(\Phi)$ will be the injectivity set of the operator $\Phi$, which turns out to be the largest subset for which Problem~\ref{prob:BPgeneralPhi} is non-trivial. Problem~\ref{prob:BPgeneralPhi} has been considered widely, e.g., for centroid bodies \cite{Lutwak1990}, for $L_p$-intersection bodies \cite{Yaskin2006}, for Minkowski valuations compatible with rotations \cite{Schuster2008}, for complex intersection bodies \cite{Koldobsky2008}, replacing real by complex hyperplanes in \eqref{eq:BPproblem} and origin-symmetry by $\unitsurf{1}$-invariance (that is, invariance under multiplication by $c \in \unitsurf{1}\subset \CC$), as well as in a functional setting \cite{Koldobsky2023}.

\bigskip

In this work, we consider Problem~\ref{prob:BPgeneralPhi} for the recently defined analogues of projection, centroid and $L_p$-intersection body maps for complex vector spaces.

The \emph{complex projection body} map $\Pi_C$ was discovered in \cite{Abardia2011} and can be defined for every $C \in \convexbodies(\CC)$ by
\begin{align}\label{eq:defCProjBody}
	h_{\Pi_{C}K}(u)=\frac12 \int_{\unitsurf{2n-1}} h_{Cu}(v) d\surfArea{K}(v), \quad K \in \convexbodies(\CC^n), u \in \unitsurf{2n-1},
\end{align}
where $Cu=\{cu:c \in C\}\subset \CC^n$ and $\surfArea{K}$ denotes the surface area measure of $K$ (see, e.g., \cite{Schneider2014}*{Sec.~4.2}). Motivated by this construction, a \emph{complex centroid body} map was introduced in \cite{Haberl2019} and \emph{complex $L_p$-intersection body} maps in \cite{Ellmeyer2023} by
\begin{align}\label{eq:defLpCIntBody}
	\rho_{\IntBody_{C,p} K}(u)^{-p}=\int_K h_{Cu}(x)^p dx, \quad K \in \starbodiesO(\CC^n), u \in \unitsurf{2n-1},
\end{align}
where $C \in \convexbodies(\CC)$ contains the origin in its relative interior and $-\dim C < p \leq 1$ is non-zero. Let us point out that for $p=1$, \eqref{eq:defLpCIntBody} yields (up to normalization) the polar of the complex centroid body from \cite{Haberl2019}. Moreover, the real operators can be recovered by letting $C = [-1,1] \subset \CC$ and $p \to -1^+$ (up to suitable renormalizations).

In the following, we consider the operators $\Pi_C$ and $\IntBody_{C,p}$ only for origin-symmetric $C \in \convexbodies(\CC)$ containing the origin in its interior. This will be a general condition in all statements below, together with the assumption that $-2<p\leq 1$ is non-zero.

As in the real setting, $\Pi_C$ and $\IntBody_{C,p}$ behave in a dual way and the roles of the bodies $K$ and $L$ in Problem~\ref{prob:BPgeneralPhi} are interchanged. In order to state the results in a unified way, we will therefore "dualize" the problem for $\Pi_C$ by considering its polar operator $\Pi_C^\circ$, $\Pi_C^\circ K = (\Pi_C K)^\circ$, where $K^\circ = \{y \in \RR^m: \langle y, x \rangle \leq 1, \forall x \in K\}$ denotes the polar body of $K \in \convexbodies(\RR^m)$ containing the origin in its interior. As polarity is order-reversing, this interchanges $K$ and $L$ in Problem~\ref{prob:BPgeneralPhi}, but has no further effects.

Moreover, note that for $p > 0$, the volume inequality that we expect in Problem~\ref{prob:BPgeneralPhi} is reversed. We will take care of this fact by multiplying both sides with $-p<0$ and set $p=1$ when $\Phi = \Pi_C^\circ$. Let $\DD\subset \CC$ denote the unit disc. Our main result states the following.
\begin{MainTheorem}\label{mthm:stateComplBP}
	Let $\Phi \in \{\Pi_{C}^\circ,\CIntBody{C}{p}\}$. If $p<0$, then the answer to Problem~\ref{prob:BPgeneralPhi} is affirmative for $n=2$ and $\Phi=\IntBody_{\DD, p}$, and negative for $n \geq 3$. If $p>0$, then the answer to Problem~\ref{prob:BPgeneralPhi} with reversed volume inequality is negative for all $n \geq 2$.
\end{MainTheorem}
Let us point out that Theorem~\ref{mthm:stateComplBP} implies a similar statement for the complex centroid body defined in \cite{Haberl2019} (see Remark~\ref{rem:compCentroidBody} below).

Our proof of Theorem~\ref{mthm:stateComplBP} follows a strategy similar to \cites{Gardner1999,Yaskin2006,Koldobsky2004, Ryabogin2004} for real Busemann--Petty type problems. In particular, it relies heavily on a connection to the spherical Fourier transform. Indeed, all operators $\Phi \in \{\Pi_{C}^\circ,\CIntBody{C}{p}\}$ can be represented in terms of the operator $\JOp{C}{p}: C(\unitsurf{2n-1}) \to C(\unitsurf{2n-1})$, defined for $f \in C(\unitsurf{2n-1})$ by (see \cite{Ellmeyer2023})
\begin{align}\label{eq:defJOp}
	(\JOp{C}{p} f)(u) = \int_{\unitsurf{2n-1}} h_{C}(v \cdot u)^p f(v) dv, \quad u \in \unitsurf{2n-1},
\end{align}
where $v \cdot u = \langle v, u\rangle + i\langle -iv, u\rangle$ denotes the Hermitian inner product on $\CC^n$. This follows directly from \eqref{eq:defCProjBody} resp.\ \eqref{eq:defLpCIntBody} and the relation $h_{Cu}(v) = h_C(v \cdot u)$. It is our second main result that the map $\JOp{C}{p}$ on even functions can be decomposed into the (distributional) spherical Fourier transform $\sphericalFourier{q}$, defined for even $\varphi \in C^\infty(\unitsurf{m-1})$ by
\begin{align}\label{eq:defSphFourier}
	(\sphericalFourier{q} \varphi)(u) = (\Fourier \varphi_q)(u), \quad u \in \unitsurf{m-1},
\end{align}
where $\varphi_q$ denotes the $q$-homogeneous extension of $\varphi \in C^\infty(\unitsurf{m-1})$ and $\Fourier$ the classical Fourier transform (see Section~\ref{sec:JOp} below), and an integral operator on the body $C$.
\begin{MainTheorem}\label{mthm:JOp}
	Let $C \in \convexbodies(\CC)$ be origin-symmetric containing the origin in its interior and let $-2<p\leq 1$ be non-zero. Then there exists a finite even Borel measure $\nu_{C,p}$ on $\unitsurf{1}$, such that $-p\, \nu_{C,p}\geq 0$, and
	\begin{align}\label{eq:JopBySphFour}
		\JOp{C}{p}\varphi(u)=\frac{1}{(2\pi)^2}\int_{\unitsurf{1}}(\sphericalFourier{-2n-p}\varphi)(cu) d\nu_{C,p}(c), \quad u \in \unitsurf{2n-1},
	\end{align}
	for all even $\varphi \in C^\infty(\unitsurf{2n-1})$.
\end{MainTheorem}
For a more explicit description of the measure $\nu_{C,p}$, we refer to Section~\ref{sec:JOp}. Let us note here that Theorem~\ref{mthm:JOp} extends the well-known connection between the $L_p$-cosine transform and the spherical Fourier transform (see, e.g., \cite{Goodey2009}), which corresponds to the limiting case $C = [-1,1]$ with $\nu_{C,p} = c_p (\delta_{1} + \delta_{-1})$ for some $c_p \in \RR$.

\medskip
Using Theorem~\ref{mthm:JOp}, the proof of Theorem~\ref{mthm:stateComplBP} decomposes into the following steps which, for the reader's convenience, we will explain here in detail, including the exact statements proven in each step. Throughout, we assume $n \geq 2$.

\subsection{Determination of $\convexbodies(\Phi)$}\label{sec:detInjSet}
The first step is to determine the maximal set $\convexbodies(\Phi)$ for which Problem~\ref{prob:BPgeneralPhi} is non-trivial. As mentioned above, it is natural to define $\convexbodies(\Phi)$ as the injectivity set of $\Phi$. More precisely, let $\mathcal{M}(\unitsurf{2n-1})$ denote the set of all (signed) finite Borel measures on $\unitsurf{2n-1}$ and denote by
\begin{align*}
	\mathrm{Inj}(\JOp{C}{p}) = \{\mu  \in \mathcal{M}(\unitsurf{2n-1}): \forall k,l \in \NN: \pi_{k,l}(\JOp{C}{p}\mu) = 0 \implies \pi_{k,l}\mu = 0\}
\end{align*}
the injectivity space of $\JOp{C}{p}$, where $\pi_{k,l}$ denotes the orthogonal projection onto the space of spherical harmonics of bi-degree $(k,l) \in \NN \times \NN$ (see Section~\ref{sec:JOp} for details). The injectivity sets of $\Pi_C$ and $\IntBody_{C, p}$ are defined by
\begin{align}
	\convexbodies(\CompProj_{C}) = \{K \in \convexbodies(\CC^n): \interior K \neq \emptyset, S_K \in \mathrm{Inj}(\JOp{C}{1})\}, \label{eq:InjSetPiC}\\
	\convexbodies(\CIntBody{C}{p}) = \{K \in \convexbodies(\CC^n): \interior K \neq \emptyset, \rho_{K}^{2n+p} \in \mathrm{Inj}(\JOp{C}{p})\}. \label{eq:InjSetIcp}
\end{align}
In particular, if $C$ is origin-symmetric, then every $K \in \convexbodies(\Phi)$ is origin-symmetric, and all $\unitsurf{1}$-invariant convex bodies with non-empty interior are contained both in $\convexbodies(\Pi_{C})$ and $\convexbodies(\IntBody_{C, p})$.

Outside the injectivity sets, a perturbation argument using a specific spherical harmonic shows that the answer to Problem~\ref{prob:BPgeneralPhi} is negative in general. It is therefore necessary and reasonable to restrict the study of Problem~\ref{prob:BPgeneralPhi} to bodies in $\convexbodies(\Phi)$.
\begin{proposition}\label{prop:injecset}
	Let $\Phi \in \{\CompProj_{C}^\circ,\CIntBody{C}{p}\}$ and suppose that $L \in \convexbodies(\CC^n)$ is smooth with positive curvature. If $\convexbodies(\Phi)$ is a strict subset of the set of convex bodies with non-empty interiors, then there is $K \notin \convexbodies(\Phi)$ with non-empty interior, such that
	\begin{align*}
		\Phi K=\Phi L \quad \quad \text{but} \quad \quad -p\Vol{2n}{K}>-p\Vol{2n}{L}.
	\end{align*}
\end{proposition}

\subsection{Affirmative Cases}\label{sec:affirmCases}
The next step is to confirm the statement of Problem~\ref{prob:BPgeneralPhi} in special cases, corresponding to intersection resp.\ projection bodies in the Busemann--Petty resp.\ Shephard problem. Differently to the real problems, here, the answer is affirmative for bodies in the image of the \emph{conjugated} operator $\Phi^\#$, where $\IntBody_{C,p}^\# = \IntBody_{\overline{C},p}$ and $\Pi_{C}^\# = \Pi_{\overline{C}}$ with $\overline{\,\cdot\,}$ denoting complex conjugation. If $C = \overline{C}$, in particular, in the limit $C \to [-1,1]$, this is consistent with the real case.

We write $\image \Phi^\#$ for $\Pi_{\overline{C}}(\convexbodies)$ resp.\ $\IntBody_{\overline{C}, p}(\starbodiesO)$, where
\begin{align*}
	\Pi_{C}(\convexbodies) = \{K \in \convexbodies(\Phi): \,\exists L \in \convexbodies(\CC^n): \Pi_C L = K\},\\
	\IntBody_{C, p}(\starbodiesO) = \{K \in \convexbodies(\Phi): \,\exists L \in \starbodiesO(\CC^n): \IntBody_{C, p} L = K\}.
\end{align*}
Let us point out here that we consider the image of all convex resp.\ star bodies, that is, we do not consider $\Phi$ as operator on $\convexbodies(\Phi)$ only. The statement then reads

\begin{proposition}\label{prop:AffirmImage}
	Let  $\Phi \in \{\CompProj_{C}^\circ,\CIntBody{C}{p}\}$. If $K \in \image \Phi^\#$ and $L \in \convexbodies(\Phi)$, then
	\begin{align*}
		\Phi K \subseteq \Phi L \quad \implies \quad -pV_{2n}(K) \leq -pV_{2n}(L).
	\end{align*}
	Moreover, in this case, if $\Phi K \subseteq \Phi L$, then $\Vol{2n}{K}=\Vol{2n}{L}$ holds if and only if $K$ is a translate of $L$, when $\Phi=\CompProj_{C}^\circ$, and only if $K=L$, when $\Phi=\CIntBody{C}{p}$.
\end{proposition}
Proposition~\ref{prop:AffirmImage} is proved using an adjointness property of $\Pi_{C}$ resp.\ $\IntBody_{C, p}$ with respect to mixed volumes resp.\ dual mixed volumes, similar to the real case. In this sense, $\Phi^\#$ can be interpreted as the adjoint of $\Phi$ (while $\JOp{\overline{C}}{p}$ is the $L_2$-adjoint of $\JOp{C}{p}$).

\subsection{Outside the image}\label{sec:outsideIm}
To understand the image of $\Pi_{C}^\circ$ resp.\ $\IntBody_{C, p}$, we will use the notion of a body embedding into $L_p$ (see Definition~\ref{def:embedIsomLpNeg} below), which already proved very useful for $L_p$-intersection bodies. Indeed, we will show that all bodies in $\Pi_{C}^\circ(\convexbodies)$ resp.\ $\IntBody_{C, p}(\starbodiesO)$ embed into $L_p$.

Applying a known characterization of embeddability by the spherical Fourier transform, a perturbation argument then shows that Problem~\ref{prob:BPgeneralPhi}, restricted to $\unitsurf{1}$-invariant bodies, has a negative answer outside the images. Let us point out here that, in contrast to the usual treatment of real Busemann--Petty problems, we do not extend the maps $\IntBody_{C, p}$ to measures. Therefore, technically, we have to consider the (weak) closure of the images (see \Cref{prop:counterexIDp} and \Cref{prop:counterexPiD} below).
\begin{proposition}\label{prop:counterex}
	Let $\Phi \in \{\CompProj_{C}^\circ,\CIntBody{C}{p}\}$. If there is an $\unitsurf{1}$-invariant body in $\convexbodies(\Phi)$, which is not in the (weak) closure of $\Pi_{C}^\circ(\convexbodies)\cap \convexbodies(\Pi_C)$ resp.\ $\IntBody_{C,p}(\starbodiesO)\cap \convexbodies(\IntBody_{C,p})$, then there exist $\unitsurf{1}$-invariant bodies $K, L \in \convexbodies(\Phi)$ such that
	\begin{align*}
		\Phi K \subseteq \Phi L \quad \quad \text{but} \quad \quad  -p\Vol{2n}{K}>-p\Vol{2n}{L}.
	\end{align*}
\end{proposition}

\subsection{Counter examples in higher dimensions}\label{sec:counterexHidim}
Proposition~\ref{prop:counterex} allows to give $\unitsurf{1}$-invariant counter examples to Problem~\ref{prob:BPgeneralPhi} in higher dimensions. These are known examples of bodies which do not embed into $L_p$ (see \cite{Koldobsky2008}).
\begin{proposition}\label{prop:counterExHidim}
	Let $\Phi \in \{\CompProj_{C}^\circ,\CIntBody{C}{p}\}$. If $n \geq 3$, then there exists an $\unitsurf{1}$-invariant body $K \in \convexbodies(\Phi)$, which is not in the (weak) closure of $\Pi_{C}^\circ(\convexbodies)\cap \convexbodies(\Pi_C)$ resp.\ $\IntBody_{C,p}(\starbodiesO)\cap \convexbodies(\IntBody_{C,p})$. If $p>0$, then the same is true also for $n=2$.
\end{proposition}

\subsection{Affirmative answer in low dimension}\label{sec:affAnsLoDim}
The final step is to give an affirmative answer to Problem~\ref{prob:BPgeneralPhi} in the remaining (complex) dimension $n=2$. As we describe the image of $\IntBody_{C, p}$ only in complete detail for $C = \DD$, we prove the following.
\begin{proposition}\label{MainThmB}
	If $\Phi=\CIntBody{\DD}{p}$ and $-2<p<0$, then the answer to Problem~\ref{prob:BPgeneralPhi} is affirmative in $\CC^2$.
\end{proposition}

\bigskip

\subsection*{Plan of the article} We start with an analysis of $\JOp{C}{p}$ in terms of spherical harmonics and give a proof of Theorem~\ref{mthm:JOp} in Section~\ref{sec:JOp}. Then we prove Theorem~\ref{mthm:stateComplBP} for the complex $L_p$-intersection bodies $\IntBody_{C, p}$ in Section~\ref{sec:intbodies}, and for the complex projection body in Section~\ref{sec:projbodies}, following the outlined proof strategy. The necessary background will be given wherever it is required first.

\section{Analysis of the operator $\JOp{C}{p}$}\label{sec:JOp}
In this section, we study the operator $\JOp{C}{p}$ using the theory of spherical harmonics in complex vector spaces, and prove Theorem~\ref{mthm:JOp}. We will start by recalling some background on spherical harmonics. For further details, we refer to \cites{Abardia2015, Groemer1996,Quinto1987, Rudin2008}, as well as to the references therein. We will orient ourselves at the presentation in \cite{Abardia2015}. For basic properties of the operator $\JOp{C}{p}$, we refer to \cite{Ellmeyer2023}.

\medskip

First, recall that a spherical harmonic on $\unitsurf{2n-1}$ is the restriction of a harmonic polynomial on $\CC^n = \RR^{2n}$ to $\unitsurf{2n-1}$. Under the natural action of the orthogonal group $\OO(2n)$, the space $\mathcal{H}^{2n}$ of all spherical harmonics decomposes into a direct sum of its $\OO(2n)$-irreducible subspaces $\mathcal{H}_k^{2n}$, where $\mathcal{H}_k^{2n}$ consists of all restrictions of $k$-homogeneous harmonic polynomials. As the unitary group $\UU(n)$ is a subrepresentation of $\OO(2n)$, we can further decompose $\mathcal{H}_k^{2n}$ into its $\UU(n)$-irreducible components $\mathcal{H}_{l,m}^{2n}$, $l+m=k$, consisting of spherical harmonics of bi-degree $(l,m)$. Here, a spherical harmonic $Y$ has bi-degree $(l,m) \in \NN\times\NN$, if $Y(cu) = c^l \overline{c}^m Y(u)$, $u \in \unitsurf{2n-1}, c\in\unitsurf{1}$.

Note that the decomposition of $\mathcal{H}^{2n}$ into the spaces $\mathcal{H}_{l,m}^{2n}$ is orthogonal with respect to the standard $L_2$-inner product. We denote the orthogonal projection from $L_2(\unitsurf{2n-1})$ (or measures on $\unitsurf{2n-1}$) onto $\mathcal{H}_{l,m}^{2n}$ by $\pi_{l,m}$. Then every $f \in L_2(\unitsurf{2n-1})$ is uniquely determined by its harmonic components $\pi_{l,m} f$, $l,m \in \NN$.

\medskip

Next, a transform $T: C(\unitsurf{2n-1}) \to C(\unitsurf{2n-1})$ is called a \emph{multiplier transform}, if there exist $\lambda_{k,l}[T] \in \CC$, $k,l \in \NN$, called multipliers, such that
\begin{align*}
	\pi_{k,l} (Tf) = \lambda_{k,l}[T] \pi_{k,l} f, \quad \forall f \in C(\unitsurf{2n-1}).
\end{align*}
Clearly, $T$ is injective if and only if all multipliers are non-zero. Moreover, two multiplier transforms always commute. The operator $\JOp{C}{p}$ and the spherical Fourier transform are examples of multiplier transforms.

\begin{example}\hfill
\begin{enumerate}[label=\roman*)]
\item The multipliers of $\JOp{C}{p}$ were computed by the authors in \cite[Prop.~3.4]{Ellmeyer2023} by the complex Funk--Hecke theorem (see \cite{Quinto1987}): Suppose that $\{0\} \neq C \in \convexbodies(\CC)$ contains the origin in its relative interior, and let $p > -\dim C$ be nonzero. Then the multipliers of the transform $\JOp{C}{p}$ on $\unitsurf{2n-1}$ are given for $k,l \in \NN$ by
 	\begin{align}\label{eq:multJOp}
 		 \lambda_{k,l}[\JOp{C}{p}] = \begin{cases}
                               c_0[h_C^p] 2\alpha_{k,l}^{(n,p)}, & k = l,\\
                               c_{l-k}[h_C^p] \alpha_{k,l}^{(n,p)}, & k \neq l,
                              \end{cases}
 	\end{align}
 where
	\begin{align*}
		\alpha_{k,l}^{(n,p)} = \pi^n\frac{\Gamma\left(\frac{p+k-l}{2}+1\right)\Gamma\left(\frac{p-k+l}{2}+1\right)}{\Gamma\left(\frac{p+k+l}{2}+n\right)\Gamma\left(\frac{p-k-l}{2}+1\right)},
	\end{align*}
and $c_m[f]$ denotes the $m$-th Fourier coefficient of $f \in C(\unitsurf{1})$, i.e.
	\begin{align*}
		c_m[f]=\begin{cases}
                               \frac{1}{2\pi} \int_{\unitsurf{1}}f(c) dc, & m = 0,\\
                               \frac{1}{\pi}\int_{\unitsurf{1}} f(c)c^{m} dc, & m \neq 0.
                              \end{cases}
	\end{align*}
\item The spherical Fourier transform $\sphericalFourier{p}$ of degree $p$ on $\unitsurf{2n-1}$, $-2n<p <0$, was defined in the introduction by $\sphericalFourier{p} \varphi = (\Fourier \varphi_p)|_{\unitsurf{2n-1}}$, where $\varphi_p$ is the $p$-homoge\-neous extension of an even function $\varphi \in C^\infty(\unitsurf{2n-1})$. Note that, for this definition, it is important that $\Fourier \varphi_p$ is again a smooth function (see \cite{Goodey2009}, it is homogeneous of degree $-2n-p$) so that $\sphericalFourier{p}$ defines a linear operator on even smooth functions. It can be extended to even distributions using that $\sphericalFourier{p}$ is self-adjoint.

As $\sphericalFourier{p}$ intertwines the $\OO(2n)$-action on smooth functions, it acts as a multiplier transform on the spaces $\mathcal{H}_{k,l}^{2n}$, where $k+l$ is even. The multipliers were computed in \cite[Lem.~3.4]{Goodey2009} (however, using a different parametrization),
	\begin{align}\label{eq:FpMult}
		\lambda_{k,l}\left[\sphericalFourier{p}\right]=(-1)^{\frac{k+l}{2}}2^{2n+p}\pi^{n}\frac{\GAMMA{\frac{k+l+p}{2} + n}}{\GAMMA{\frac{k+l-p}{2}}}.
	\end{align}
This formula extends analytically to all $p \in \CC$ which are not even integers. Moreover, we can deduce from \eqref{eq:FpMult} that the inverse of $\sphericalFourier{p}$ is given by
\begin{align}\label{eq:FourierInverse}
	\sphericalFourier{-2n-p}\sphericalFourier{p}\varphi=(2\pi)^n \varphi.
\end{align}
\end{enumerate}
\end{example}

\medskip

Combining \eqref{eq:multJOp} and \eqref{eq:FpMult}, we deduce that \eqref{eq:JopBySphFour} holds, if the integral transform defined by $\nu_{C,p}$ has multipliers $4\pi^2\lambda_{k,l}[\JOp{C}{p}]/\lambda_{k,l}[\sphericalFourier{p}]$. It remains to find a measure $\nu_{C,p}$ with this property and an appropriate sign. For this reason, we first describe the multipliers of operators $\TT{\mu}$ of the form
\begin{align}\label{eq:operatorTmu}
	\TT{\mu}f(u)=\int_{\unitsurf{1}} f(cu) d\mu(c), \quad u \in \unitsurf{2n-1}, f \in C(\unitsurf{2n-1}),
\end{align}
where $\mu$ is a finite Borel measure on $\unitsurf{1}$. In the following, we denote by $\innerProduct{\cdot, \cdot}$ the complex $L_2$-inner product resp.\ dual pairing on $\unitsurf{2n-1}$.
\begin{lemma}\label{lem:Tmult}
Let $\mu$ be a finite real Borel measure on $\unitsurf{1}$ and let $f,g \in C(\unitsurf{2n-1})$.
\begin{enumerate}[label=\roman*)]
\item\label{it:Tmult1} Then $\left\langle \TT{\mu}f, g \right\rangle=\left\langle f,\TT{\mu^{\#}}g \right\rangle$, where $\mu^{\#}$ is the push-forward measure of $\mu$ by complex conjugation.
\item\label{it:Tmult2} The multipliers $\lambda_{k,l}[T_\mu]$, $k,l \in \NN$, of $\TT{\mu}$ are given by 
	\begin{align*}
		\lambda_{k,l}[\TT{\mu}]=\begin{cases}
2\pi c_{0}[\mu] \quad &k=l,\\
\pi c_{l-k}[\mu] \quad &k\neq l.
		\end{cases}
	\end{align*}
In particular, $\lambda_{k,l}[\TT{\mu^{\#}}]=\pi c_{k-l}[\mu]$ for $k\neq l \in \NN$.
\end{enumerate}
 \end{lemma}
Let us note that Lemma~\ref{lem:Tmult} generalizes a result from \cite{Abardia2015} for $\mu=S_C$, $C \in \convexbodies(\CC)$. By \ref{it:Tmult1}, $\TT{\mu}$ extends to measures.
\begin{proof}
\ref{it:Tmult1} follows by a direct computation
	\begin{align*}
		&\left\langle \TT{\mu}f, g \right\rangle=\int_{\unitsurf{2n-1}}\int_{\unitsurf{1}}f(cu)d\mu(c) \overline{g(u)} du=\int_{\unitsurf{2n-1}}f(u)\int_{\unitsurf{1}}\overline{g(\overline{c}u)}d\mu(c)  du\\
		=&\int_{\unitsurf{2n-1}}f(u)\overline{\int_{\unitsurf{1}}g(\overline{c}u)d\mu(c)}  du=\left\langle f, \TT{\mu^{\#}}g \right\rangle.
	\end{align*}

For \ref{it:Tmult2}, we use the $\UU(n)$-invariance of the spherical Lebesgue measure and Fubini's theorem, to obtain for $Y_{k,l}\in \mathcal{H}_{k,l}^{2n}$, $k \neq l$,
	\begin{align*}
		\left\langle Y_{k,l},\TT{\mu}f \right\rangle&=\left\langle \TT{\mu^{\#}}Y_{k,l},f \right\rangle=
\int_{\unitsurf{2n-1}} \int_{\unitsurf{1}} Y_{k,l}(\overline{c}u) d\mu(c) \overline{f(u)} du\\
&=\int_{\unitsurf{2n-1}} \int_{\unitsurf{1}} \overline{c}^{k-l}Y_{k,l}(u) d\mu(c) \overline{f(u)} du = \pi c_{l-k}[\mu]\int_{\unitsurf{2n-1}} Y_{k,l}(u)\overline{f(u)} du \\
&=\pi c_{l-k}[\mu]\left\langle Y_{k,l},f \right\rangle,
	\end{align*}
which yields the claim in the case $k \neq l$. The case $k=l$ is similar.
\end{proof}
%

\medskip

To find $\nu_{C,p}$, we will use the notion of \emph{embedding into $L_p$}.
\begin{definition}\cite[Def.~6.14]{Koldobsky2005}\label{def:embedIsomLpNeg}
Suppose that $K \in \starbodiesO (\RR^n)$ is origin-symmetric. Then the space $(\RR^n, \|\cdot \|_{K})$ is said to \emph{embed in $L_p$}
\begin{enumerate}[label=\roman*)]
\item  for $-n<p<0$, if there exists a finite Borel measure~$\mu $ on $\unitsurf{n-1}$ such that
	\begin{align}\label{eq:embedIsomLpNeg}
		\int _{\RR^{n}} \|x\|_{K}^{p} \phi (x) dx = \int _{\unitsurf{n-1}}\left ( \int _{0}^{\infty }r^{-p-1} ({\bf{\mathrm{F}}}\phi)(ru) dr \right ) d\mu (u),
	\end{align}
for every even Schwartz function $\phi $ on $\RR^n$ and
\item for $p>0$ that is not an even integer, if there exists a finite Borel measure~$\mu $ on $\unitsurf{n-1}$ such that
	\begin{align}
		\int _{\RR^n} \|x\|_{K}^{p} \phi (x) dx = \frac{1}{\GAMMA{-\frac p2}}\int _{\unitsurf{n-1}}\left ( \int _{0}^{\infty }r^{-p-1} ({\bf{\mathrm{F}}}\phi)(ru) dr \right ) d\mu (u),
	\end{align}
for every even Schwartz function $\phi $ on $\RR^n$, whose Fourier transform is supported outside of the origin.
\end{enumerate}
We denote the set of all star bodies that embed into $L_p$ by $\{\hookrightarrow L_p\}$.
\end{definition}

An alternative characterization of when a body embeds into $L_p$ was given in \cite{Koldobsky1992}, see also \cite{Koldobsky2005}*{Thm.~6.10, Thm.~6.15} and \cite{Rubin2008}*{Thm.~5.2, Prop.~5.4} (formulated in terms of the spherical Fourier transform used here).

\begin{proposition}\label{prop:embedLp}
Let $K\in \starbodiesO(\RR^n)$ be origin-symmetric and let $-n<p$ be non-zero. Then $K$ embeds into $L_p$ if and only if
	\begin{align*}
		\frac{1}{\GAMMA{-\frac p2}}\sphericalFourier{p}\rho_K^{-p} \geq 0,
	\end{align*}
in the sense of distributions. In particular, $\frac{1}{\GAMMA{-\frac p2}}\sphericalFourier{p}\rho_K^{-p}$ can be represented by a positive measure.
\end{proposition}

We can now use that every suitable origin-symmetric body $C \subset \CC$ embeds into $L_p$ and, hence, by Proposition~\ref{prop:embedLp}, we obtain a measure with the correct sign.
\begin{proposition}\label{prop:FphCmeasure}
Let $C \in \convexbodies(\CC)$ be origin-symmetric containing the origin in its interior, and $-2<p\leq 1$. Then the distribution $\nu_{C,p} = \sphericalFourier{p}h_C^p$ on $\unitsurf{1}$ satisfies
	\begin{align*}
		\frac{1}{\GAMMA{-\frac p2}}\nu_{C,p} \geq 0,
	\end{align*}
and therefore is a measure. The multipliers of $\TT{\nu_{C,p}}$ are given by
\begin{align}\label{eq:multTnuCp}
	\lambda_{k,l}[\TT{\nu_{C,p}}]=\begin{cases}
		2\pi k_p\GAMMA{\frac{p}{2}+1}^2c_0[h_C^p] , & k=l,\\
		\pi k_p\GAMMA{\frac{p+l-k}{2}+1}\GAMMA{\frac{p-l+k}{2}+1}c_{l-k}[h_C^p], & k\neq l,
	\end{cases}
\end{align}
where $k_p=-2^{2+p}\sin\left(\frac{p\pi}{2}\right)$.

For $p=1$, we have $\nu_{C,1}=-2\pi \mathcal{S}(iC, \cdot)$ and for $p=-1$, we have $d\nu_{C,-1}= 2\pi\rho_{iC^\circ} dc$.
	
\end{proposition}

\begin{proof}
By \cite{Koldobsky2005}*{Cor.~6.7\&6.8, Thm.~6.17}, every origin-symmetric convex body $C$ with non-empty interior embeds into $L_p$ for every non-zero $-2<p\leq 1$. Now, \Cref{prop:embedLp} implies that $\sphericalFourier{p}h_C^p=\sphericalFourier{p}\rho_{C^\circ}^{-p}$ is up to the sign of $\GAMMA{-\frac p2}$ a positive and even measure, yielding the first claim.

Next, noting that $\mathcal{H}_k^2$ is spanned by the maps $c \mapsto c^k, c^{-k}$, we have for even $k \neq 0$
\begin{align*}
	\pi c_k[\sphericalFourier{p}h_C^p] = \langle \pi_{k,0} \sphericalFourier{p}h_C^p, c^k \rangle = \lambda_{k,0}[\sphericalFourier{p}] \langle \pi_{k,0} h_C^p, c^k \rangle = \pi \lambda_{k,0}[\sphericalFourier{p}] c_k[h_C^p],
\end{align*}
and thus, by \eqref{eq:FpMult} and a similar computation for $k=0$,
\begin{align}\label{eq:prfDefNucpMult}
	c_k[\sphericalFourier{p}h_C^p] = \lambda_{k,0}[\sphericalFourier{p}] c_k[h_C^p] = (-1)^{\frac{k}{2}}2^{2+p}\pi\frac{\GAMMA{\frac{k+p}{2} + 1}}{\GAMMA{\frac{k-p}{2}}} c_k[h_C^p].
\end{align}
Applying Euler's reflection formula, we obtain for even $k$,
\begin{align*}
	c_k[\sphericalFourier{p}h_C^p]=-2^{2+p}\sin\left(\frac{p\pi}{2}\right)\GAMMA{\frac{p+k}{2}+1}\GAMMA{\frac{p-k}{2}+1}c_k[h_C^p],
\end{align*}
which, by Lemma~\ref{lem:Tmult}, yields the second claim as the multipliers for $k+l$ odd vanish.

For $p=1$ and even $k$, \eqref{eq:prfDefNucpMult} reduces to
\begin{align*}
	c_{k}[\nu_{C,1}]=(-1)^{\frac k2}8\pi\frac{\GAMMA{\frac{k+3}{2}}}{\GAMMA{\frac{k-1}{2}}}c_k[h_C]=2\pi(-1)^{\frac k2} (k^2-1)c_k[h_C].
\end{align*}
As, by \cite{Abardia2015}*{Lem.~4.6}, the Fourier coefficients of $S_{C}(\cdot)$ are given by $(1-k^2)c_k[h_C]$, this yields the claim for $p=1$. The statement for $p=-1$ follows directly.
\end{proof}

\begin{corollary}\label{cor:adjCkonjcompl}
	Let $C \in \convexbodies(\CC)$ be origin-symmetric, $0 \in \interior C$, and $-2<p\leq 1$ non-zero. Then $\nu_{\overline{C},p} = \nu_{C,p}^\#$ and, thus, $\TT{\nu_{\overline{C},p}}$ and $\TT{\nu_{C,p}}$ are adjoint.
\end{corollary}
\begin{proof}
	As $c_k[\mu^\#] = \overline{c_k[\mu]}$ for a real measure $\mu$, and $c_k[h_{\overline{C}}^p] = \overline{c_k[h_C^p]}$, the first claim follows from \eqref{eq:prfDefNucpMult}. The second claim follows from Lemma~\ref{lem:Tmult}\ref{it:Tmult1}.
\end{proof}

We are now ready to prove Theorem~\ref{mthm:JOp}.
\begin{proof}[Proof of Theorem~\ref{mthm:JOp}] 
	By the origin-symmetry of $C$, $c_{l-k}[h_C^p]=0$ when $k+l$ is odd, and thus $\lambda_{k,l}[\JOp{C}{p}] = 0$, that is, we only have to consider $k+l$ even. By \eqref{eq:multJOp} and \eqref{eq:FpMult}, we obtain using $k_p = -2^{2+p}\sin(p\pi/2)$
	\begin{align*}
		\lambda_{k,l}[\JOp{C}{p}]=\begin{cases}
			\lambda_{k,l}[\sphericalFourier{-2n-p}] \frac{k_p}{2\pi}\GAMMA{\frac p2 +1}^2 c_0[h_C^p] , & k = l,\\
			\lambda_{k,l}[\sphericalFourier{-2n-p}] \frac{k_p}{4\pi}\GAMMA{\frac {p+k-l}{2} +1}\GAMMA{\frac {p-k+l}{2} +1}c_{l-k}[h_C^p], & k\neq l.
		\end{cases}
	\end{align*}
	Consequently, by \eqref{eq:multTnuCp}, $(2\pi)^2\lambda_{k,l}[\JOp{C}{p}] = \lambda_{k,l}[\sphericalFourier{-2n-p}]\lambda_{k,l}[T_{\nu_{C,p}}]$, yielding the claim.
\end{proof}

\medskip

Let us point out that the spherical Fourier transform maps smooth even functions to smooth even functions on $\unitsurf{2n-1}$. However, $\sphericalFourier{p} f$ for $f \in C(\unitsurf{2n-1})$, defined in the distributional sense, in general only yields a distribution. As we want to apply Theorem~\ref{mthm:JOp} in the following for radial functions of general star resp. convex bodies, we need to argue that, in this case, $\TT{\nu_{C,p}}\circ \sphericalFourier{-2n-p}$ maps continuous functions to continuous functions.
\begin{proposition}\label{prop:repIcpPicByFourier}
	Let $C \in \convexbodies(\CC)$ be origin-symmetric containing the origin in its interior, and $-2<p\leq 1$ non-zero. Then for $\nu_{C,p} = \sphericalFourier{p}h_C^p$ we have
	\begin{align*}
		\rho_{\CIntBody{C}{p} K}^{-p}=\frac{1}{(2\pi)^2(2n+p)} \TT{\nu_{C,p}}\circ\sphericalFourier{-2n-p}\rho_{K}^{2n+p},
	\end{align*}
	for every $K \in \starbodiesO(\CC^n)$ and
	\begin{align*}
		h_{\CompProj_CK}=-\frac{1}{4\pi} \TT{\surfArea{iC}}\circ\sphericalFourier{-2n-1}\surfArea{K},
	\end{align*}
	for every $K \in \convexbodies(\CC^n)$ with non-empty interior.
\end{proposition}
Before proving Proposition~\ref{prop:repIcpPicByFourier}, we note that
\begin{align}\label{eq:IcpPicByJOp}
	\rho_{\IntBody_{C,p} K}^{-p} = \frac{1}{2n+p}\JOp{C}{p}(\rho_K^{2n+p}) \quad \text{ and } \quad h_{\Pi_C K} = \frac{1}{2} \JOp{C}{1}(S_K)
\end{align}
for $K \in \starbodiesO(\CC^n)$ resp. $K \in \convexbodies(\CC^n)$ with non-empty interior. Moreover, let us point out that the order of the transforms on the right-hand side of the equations in Proposition~\ref{prop:repIcpPicByFourier} is freely interchangeable as they are all multiplier transforms.
\begin{proof}[Proof of Proposition~\ref{prop:repIcpPicByFourier}]
	First note that since $C$ is origin-symmetric, and thus $\nu_{C,p}$ is even, both sides of \eqref{eq:JopBySphFour} in Theorem~\ref{mthm:JOp} vanish for odd $\varphi \in C^\infty(\unitsurf{2n-1})$. Hence, we conclude that the statement of Theorem~\ref{mthm:JOp} holds for all $\varphi \in C^\infty(\unitsurf{2n-1})$.
	
	Next, let $\varphi \in C^\infty(\unitsurf{2n-1})$. By \eqref{eq:IcpPicByJOp}, the fact that $\JOp{\overline{C}}{p}$ is the adjoint of $\JOp{C}{p}$ and Theorem~\ref{mthm:JOp},
	\begin{align*}
		\innerProduct{\rho_{\CIntBody{C}{p} K}^{-p},\varphi}&= \frac{1}{2n+p}\innerProduct{\JOp{C}{p}\rho_K^{2n+p},\varphi}=\frac{1}{2n+p}\innerProduct{\rho_K^{2n+p},\JOp{\overline{C}}{p}\varphi} \\
		&=\frac{1}{(2\pi)^2(2n+p)}\innerProduct{\rho_K^{2n+p}, \TT{\nu_{\overline{C},p}}\circ\sphericalFourier{-2n-p}\varphi}.
	\end{align*}
	Since the involved operators are multiplier transforms, they commute and therefore
	\begin{align*}
		\innerProduct{\rho_{\CIntBody{C}{p} K}^{-p},\varphi}=\frac{1}{(2\pi)^2(2n+p)}\innerProduct{\rho_K^{2n+p}, \sphericalFourier{-2n-p}\circ\TT{\nu_{\overline{C},p}}\varphi}.
	\end{align*}
	Consequently, as $\varphi$ was arbitrary, and $\TT{\nu_{C,p}} \circ \sphericalFourier{-2n-p}$ is, by Corollary~\ref{cor:adjCkonjcompl}, the adjoint of $\sphericalFourier{-2n-p}\circ\TT{\nu_{\overline{C},p}}$, we obtain $\rho_{\CIntBody{C}{p} K}^{-p}=\frac{1}{(2\pi)^2(2n+p)} \TT{\nu_{C,p}}\circ\sphericalFourier{-2n-p}\rho_K^{2n+p}$ in the sense of distributions. However, since the left hand side is a continuous function on the sphere, the equality indeed holds as equality of continuous functions.
	
	Repeating the argument above and noting that $-2\pi h_{\Pi K} = \sphericalFourier{-2n-1}\surfArea{K}$ (\cite{Koldobsky2005}*{Thm.~8.1}) is a continuous map on $\unitsurf{2n-1}$, the second equality follows since $\nu_{C,1}=-2\pi \mathcal{S}(iC, \cdot)$, by \Cref{prop:FphCmeasure}.
\end{proof}

\section{Busemann--Petty problem for complex $L_p$-intersection bodies}\label{sec:intbodies}
In this section, we prove Theorem~\ref{mthm:stateComplBP} for the complex $L_p$-intersection body maps $\IntBody_{C,p}$ following the steps outlined in the introduction.

\subsection{\nameref{sec:detInjSet}}
First, we prove that the injectivity set $\convexbodies(\IntBody_{C,p})$ defined in \eqref{eq:InjSetIcp} is maximal in the sense that the answer to Problem~\ref{prob:BPgeneralPhi} is negative when the bodies are not in $\convexbodies(\IntBody_{C,p})$.

Before doing so, let us point out that a body $K \in \convexbodies(\CC^n)$ is in $\convexbodies(\IntBody_{C,p})$, if $\pi_{k,l}\rho_K^{2n+p} = 0$ whenever $\lambda_{k,l}[\JOp{C}{p}] = 0$ or, equivalently by \eqref{eq:multJOp}, $c_{l-k}[h_C^p] = 0$. In particular, for $K, L \in \convexbodies(\IntBody_{C,p})$, $\IntBody_{C,p}K = \IntBody_{C,p} L$ implies $K=L$. Moreover, as $c_k[h_{\overline{C}}^p] = \overline{c_k[h_C^p]}$, we directly obtain that $\convexbodies(\IntBody_{\overline{C},p}) = \convexbodies(\IntBody_{C,p})$.

In the following, we will also work with the set of $\unitsurf{1}$-invariant convex bodies with non-empty interior. As a function $f \in C(\unitsurf{2n-1})$ is $\unitsurf{1}$-invariant, if and only if $\pi_{k,l} f = 0$ for all $k \neq l$, and as $c_0[h_C^p] > 0$, all such bodies are contained in $\convexbodies(\IntBody_{C,p})$.

\bigskip

For the proof of \Cref{prop:injecset}, we will need the notion of dual mixed volume of $K,L \in \starbodiesO(\RR^m)$ which was introduced for non-zero $p\in \RR$ in \cite{Lutwak1975} as
	\begin{align*}
		\DualMixedVol{p}{K,L}=\frac{1}{m}\int_{\unitsurf{m-1}}\rho_K^{m-p}(u)\rho_L^p(u) du.
	\end{align*}
Note that we always have $\DualMixedVol{p}{K,K}=\voln_{m}(K)$. If $L \subset M$, then $p\DualMixedVol{p}{K,L} \leq p \DualMixedVol{p}{K,M}$. Moreover, the dual $L_p$-Minkowski inequality states that (see, e.g., \cite{Schneider2014}*{(9.40)}) for $0<p<m$ and $K, L \in \starbodiesO(\RR^m)$
	\begin{align*}
		\DualMixedVol{p}{K,L}\leq \Vol{m}{K}^{\frac{m-p}{m}}\Vol{m}{L}^{\frac{p}{m}},
	\end{align*}
and for $p<0$
	\begin{align*}
		\DualMixedVol{p}{K,L}\geq \Vol{m}{K}^{\frac{m-p}{m}}\Vol{m}{L}^{\frac{p}{m}}.
	\end{align*}
The inequalities can be written in the following unified way for $p<0$ or $0<p<m$,
\begin{align}\label{eq:dualLpMinkineq}
	p\DualMixedVol{p}{K,L}\leq p\Vol{m}{K}^{\frac{m-p}{m}}\Vol{m}{L}^{\frac{p}{m}}.
\end{align}
Equality holds if and only if $K$ and $L$ are dilates of each other.

\medskip

Moreover, we will use the following convexity statement for small perturbations.

\begin{lemma}[\cite{Koldobsky2005}*{p.~96}]\label{lem:convradial}
Let $L$ be a smooth convex body with strictly positive curvature and $p\in \RR\setminus\{0\}$. Then for any $\varphi \in C^\infty(\unitsurf{n-1})$, the star body $K$ defined by
	\begin{align*}
		\rho_K^{-p}=\rho_L^{-p}+\varepsilon \varphi
	\end{align*}
is convex for all $\varepsilon>0$ sufficiently small.
\end{lemma}

\begin{proof}[Proof of \Cref{prop:injecset} for $\IntBody_{C,p}$]
As the set $\convexbodies(\IntBody_{C,p})$ is a strict subset, there exist indices $(k,l) \in \NN \times \NN$ such that $c_{l-k}[h_C^p]=0$. Denoting by $\widetilde{P}_{k,l} \in \mathcal{H}_{k,l}^{2n}$ the Jacobi polynomial of bi-degree $(k,l)$ (see, e.g., \cite{Johnson1977,Quinto1987}), we define
\begin{align*}
	\rho_K^{2n+p} = \rho_L^{2n+p} + \varepsilon \widetilde{P}_{k,l}
\end{align*}
where $|\varepsilon|$ is sufficiently small such that $K \in \convexbodies(\CC^n)$ with non-empty interior (using Lemma~\ref{lem:convradial}), and the sign of $\varepsilon$ is chosen in a way such that
\begin{align}\label{eq:prfPropInjSignEps}
	-p\varepsilon\innerProduct{\widetilde{P}_{k,l},\rho_L^{-p}}\geq 0.
\end{align}
By our choice of $(k,l)$, we have $\TT{\nu_{C,p}}\widetilde{P}_{k,l}=0$ and thus, by Proposition~\ref{prop:repIcpPicByFourier} and since multiplier transforms commute, $\IntBody_{C,p}K = \IntBody_{C,p}L$.

Next, note that by \eqref{eq:prfPropInjSignEps},
\begin{align*}
	-2np\DualMixedVol{-p}{K,L} = -p\innerProduct{\rho_K^{2n+p}, \rho_{L}^{-p}} = -2np\voln_{2n}(L) - p\varepsilon\innerProduct{\widetilde{P}_{k,l} , \rho_{L}^{-p}} \geq -2np\voln_{2n}(L),
\end{align*}
and, consequently, by \eqref{eq:dualLpMinkineq},
\begin{align*}
	-p\voln_{2n}(L) \leq -p \DualMixedVol{-p}{K,L} \leq -p \voln_{2n}(K)^{\frac{2n+p}{2n}} \voln_{2n}(L)^{-\frac{p}{2n}},
\end{align*}
that is, $-p\voln_{2n}(L) \leq -p\voln_{2n}(K)$. As equality holds in \eqref{eq:dualLpMinkineq} only when $K, L$ are dilates, and, by construction, this is not the case for $K$ and $L$ here, the inequality is strict, concluding the proof.
\end{proof}

\subsection{\nameref{sec:affirmCases}}
Next, we will show that Problem~\ref{prob:BPgeneralPhi} has a positive answer whenever $K\in \convexbodies(\IntBody_{C, p})$ is in the image set $\IntBody_{\overline{C},p}(\starbodiesO)$.

\begin{proof}[Proof of Proposition~\ref{prop:AffirmImage} for $\IntBody_{C,p}$]
	By Proposition~\ref{prop:repIcpPicByFourier} and Corollary~\ref{cor:adjCkonjcompl}, we obtain for every $K, L \in \starbodiesO(\CC^n)$
	\begin{align}\label{eq:prfPrpAffImIcpEqDual}
		2n \DualMixedVol{-p}{K,\IntBody_{\overline{C},p}L} &= \frac{1}{2n+p}\innerProduct{\rho_{K}^{2n+p}, \JOp{\overline{C}}{p}\rho_L^{2n+p}} \nonumber\\
		&= \frac{1}{(2\pi)^2(2n+p)}\innerProduct{\rho_{K}^{2n+p}, \TT{\nu_{\overline{C},p}}\sphericalFourier{-2n-p}\rho_L^{2n+p}}\nonumber\\
		&= \frac{1}{(2\pi)^2(2n+p)}\innerProduct{\TT{\nu_{C,p}}\sphericalFourier{-2n-p}\rho_{K}^{2n+p}, \rho_L^{2n+p}}
		= 2n \DualMixedVol{-p}{L, \IntBody_{C,p}K}.
	\end{align}
	Next, let $K \in \IntBody_{\overline{C},p}(\starbodiesO)$, that is, $K = \IntBody_{\overline{C},p} K_0$ for some $K_0 \in \starbodiesO(\CC^n)$, and assume that $\IntBody_{C,p}K \subseteq \IntBody_{C,p}L$ for $L \in \convexbodies(\IntBody_{C,p})$. Since $-p\DualMixedVol{-p}{K_0, \cdot}$ is monotone, \eqref{eq:prfPrpAffImIcpEqDual} together with the inclusion implies
	\begin{align*}
		-p\voln_{2n}(K) &= -p \DualMixedVol{-p}{K, \IntBody_{\overline{C},p}K_0} = -p \DualMixedVol{-p}{K_0,\IntBody_{C,p}K} \\&\leq -p \DualMixedVol{-p}{K_0,\IntBody_{C,p}L} = -p\DualMixedVol{-p}{L,K}
	\end{align*}
	Applying the dual $L_p$-Minkowski inequality \eqref{eq:dualLpMinkineq},
	\begin{align*}
		-p\voln_{2n}(K) \leq -p\DualMixedVol{-p}{L,K} \leq -p\Vol{2n}{L}^{\frac{2n+p}{2n}}\Vol{2n}{K}^{-\frac{p}{2n}},
	\end{align*}
	that is, $-p\voln_{2n}(K) \leq -p\voln_{2n}(L)$, as claimed.
	
	The statement when $\voln_{2n}(K) = \voln_{2n}(L)$ follows directly from the equality cases of the dual $L_p$-Minkowski inequality.
\end{proof}

\subsection{\nameref{sec:outsideIm}}
In the next step, we show that, if there are bodies that are not in the image of $\IntBody_{C, p}$, then Problem~\ref{prob:BPgeneralPhi} has a negative answer. To this end, we will give a characterization of $\IntBody_{C,p}(\starbodiesO)$ in terms of the spherical Fourier transform. Indeed, we first show that every body in the image embeds into $L_p$, that is, by \cite{Rubin2008}*{Prop.~5.4}, is a real $L_p$-intersection body.

\begin{proposition}\label{prop:CIntEmbedLp}
Let $C \in \convexbodies(\CC)$ be origin-symmetric with non-empty interior, and $-2<p\leq 1$ be non-zero. If $K\in \imgSetPhi{\CIntBody{C}{p}}{\starbodiesO}$, then 
	\begin{align*}
		\frac{1}{\GAMMA{-\frac p2}}\sphericalFourier{p}\rho_K^{-p}\geq 0. 
	\end{align*}
In particular, $\imgSetPhi{\CIntBody{C}{p}}{\starbodiesO} \subseteq \{\hookrightarrow L_p\} \cap \{K \in \starbodiesO(\CC^n): \rho_K^{-p} \in \image\TT{\nu_{C,p}}\}$.
\end{proposition}
\begin{proof}
Let $K_0 \in\starbodiesO(\CC^n)$ such that $K=\IntBody_{C,p}K_0$. By Proposition~\ref{prop:repIcpPicByFourier} and since multiplier transforms commute, we get
	\begin{align*}
		\sphericalFourier{p}\rho_{K}^{-p}=\sphericalFourier{p}\rho_{\IntBody_{C,p}K_0}^{-p}=\frac{(2\pi)^{-2}}{(2n+p)}\sphericalFourier{p}\circ\sphericalFourier{-2n-p}\circ \TT{\nu_{C,p}}\rho_{K_0}^{2n+p}
=\frac{(2\pi)^{2n-2}}{2n+p}\TT{\nu_{C,p}}\rho_{K_0}^{2n+p}.
	\end{align*}
By \Cref{prop:FphCmeasure}, $\frac{1}{\GAMMA{-\frac p2}}\nu_{C,p}$ is a positive measure and, hence,
	\begin{align*}
		\frac{1}{\GAMMA{-\frac p2}}\sphericalFourier{p}\rho_{K}^{-p}=\frac{(2\pi)^{2n-2}}{(2n+p)\GAMMA{-\frac p2}}\TT{\nu_{C,p}}\rho_{K_0}^{2n+p} \geq 0,
	\end{align*}
which by \Cref{prop:embedLp} is equivalent to $K$ embedding into $L_p$.
\end{proof}

Note that, in general, the image of $\TT{\nu_{C,p}}$ can be quite complicated and is only described by spherical harmonics. However, as it turns out, if the statement of Problem~\ref{prob:BPgeneralPhi} is negative, it already fails when restricted to $\unitsurf{1}$-invariant bodies. In the following, we will therefore concentrate on $\unitsurf{1}$-invariant bodies, which we denote by a superscript $(\cdot)^{\unitsurf{1}}$, that is, e.g., $\IntBody_{C, p}(\starbodiesO)^{\unitsurf{1}}$.

Let us point out that, if $C = \DD$, then $\image\TT{\nu_{\DD,p}} = C(\unitsurf{2n-1})^{\unitsurf{1}}$, as can be seen, e.g., by \eqref{eq:multTnuCp}. Next, we show that on $\unitsurf{1}$-invariant bodies, the image of $\IntBody_{C, p}$ does not depend on $C$ and coincides (up to closure) with $\unitsurf{1}$-invariant bodies embedding into $L_p$.
\begin{proposition}\label{prop:S1invimg}
Let $C\in \convexbodies(\CC)$ be origin-symmetric with non-empty interior, and $-2<p\leq 1$ be non-zero. Then
	\begin{align*}
		\imgSetPhi{\CIntBody{C}{p}}{\starbodiesO}^{\unitsurf{1}}= \imgSetPhi{\CIntBody{\DD}{p}}{\starbodiesO} \quad \text{and} \quad \cl_w\left(\imgSetPhi{\CIntBody{\DD}{p}}{\starbodiesO}\right)=\{\hookrightarrow L_p\}^{\unitsurf{1}},
	\end{align*}
where we denote by $\cl_w$ the closure with respect to weak convergence of $\rho_K^{-p}$.
\end{proposition}
\begin{proof}
Let $K \in \imgSetPhi{\CIntBody{C}{p}}{\starbodiesO}^{\unitsurf{1}}$, that is, $\unitsurf{1}$-invariant and $K = \IntBody_{C,p}K_0$ for some $K_0 \in \starbodiesO(\CC^n)$. By the $\unitsurf{1}$-invariance, $\TT{\nu_{C,p}} \rho_K^{-p} = \nu_{C,p}(\unitsurf{1})\rho_{K}^{-p}$ for all $C$, and, by Proposition~\ref{prop:repIcpPicByFourier},
\begin{align*}
	\nu_{\DD,p}(\unitsurf{1})\rho_K^{-p}=\TT{\nu_{\DD,p}}\rho_K^{-p}=\frac{1}{(2\pi)^2(2n+p)}\TT{\nu_{\DD,p}}\circ \sphericalFourier{-2n-p}\circ\TT{\nu_{C,p}}\rho_{K_0}^{2n+p}.
\end{align*}
We can therefore define $\widetilde{K}_0\in\starbodiesO(\CC^n)$ by
\begin{align*}
	\rho_{\widetilde{K}_0}^{2n+p}=\frac{1}{\nu_{\DD,p}(\unitsurf{1})}\TT{\nu_{C,p}}\rho_{K_0}^{2n+p}
\end{align*}
to see that $K = \IntBody_{\DD,p} \widetilde{K}_0$ and thus $K \in \imgSetPhi{\CIntBody{\DD}{p}}{\starbodiesO}$. Let us point out here that we used that $\nu_{\DD,p}(\unitsurf{1})^{-1}\nu_{C,p} \geq 0$, by Proposition~\ref{prop:FphCmeasure}, in order to get $\rho_{\widetilde{K}_0} > 0$. Repeating the argument with $C$ and $\DD$ replaced yields the other inclusion of $\imgSetPhi{\CIntBody{C}{p}}{\starbodiesO}^{\unitsurf{1}}= \imgSetPhi{\CIntBody{\DD}{p}}{\starbodiesO}$.

Next, by Proposition~\ref{prop:CIntEmbedLp} and the comment below it, $\imgSetPhi{\CIntBody{\DD}{p}}{\starbodiesO} \subseteq \{\hookrightarrow L_p\}^{\unitsurf{1}}$. As the condition in Proposition~\ref{prop:embedLp} is closed with respect to the weak topology, it therefore remains to see that $\{\hookrightarrow L_p\}^{\unitsurf{1}} \subseteq \cl_w\left(\imgSetPhi{\CIntBody{\DD}{p}}{\starbodiesO}\right)$.

To this end, let $K \in \{\hookrightarrow L_p\}^{\unitsurf{1}}$. By Propositions~\ref{prop:embedLp} and \ref{prop:FphCmeasure}, $\mu :=\nu_{\DD,p}(\unitsurf{1})^{-1}\sphericalFourier{p}\rho_{K}^{-p}$ is a positive measure on $\unitsurf{2n-1}$. We can therefore find a sequence of positive, $\unitsurf{1}$-invariant functions $f_k \in C^\infty(\unitsurf{2n-1})$, $k \in \NN$, which convergences to $\mu$ in the weak topology. For $k \in \NN$, define $K_k \in \starbodiesO(\CC^n)$ by $\rho_{K_k}^{2n+p} = f_k$. Then, by $\unitsurf{1}$-invariance,
\begin{align*}
	(2n+p)\rho_{\IntBody_{\DD,p} K_k}^{-p} = \frac{\nu_{\DD,p}(\unitsurf{1})}{(2\pi)^2} \sphericalFourier{-2n-p}\rho_{K_k}^{2n+p} \quad\stackrel{w}{\longrightarrow}\quad  \frac{\nu_{\DD,p}(\unitsurf{1})}{(2\pi)^2}\sphericalFourier{-2n-p}\mu = (2\pi)^{2n-2}\rho_{K}^{-p},
\end{align*}
concluding the proof.
\end{proof}

Next, we will use the characterization in Proposition~\ref{prop:S1invimg} to prove Proposition~\ref{prop:counterex}. For the reader's convenience, we repeat the statement using the notation of Proposition~\ref{prop:S1invimg}.
\begin{proposition}\label{prop:counterexIDp} Let $-2<p\leq 1$ be non-zero.
	If $\cl_w(\IntBody_{C,p}(\starbodiesO)) \cap \convexbodies(\IntBody_{C,p})^{\unitsurf{1}}$ is a strict subset of $\convexbodies(\IntBody_{C,p})^{\unitsurf{1}}$, then there exist $\unitsurf{1}$-invariant $K, L \in \convexbodies(\IntBody_{C, p})$ such that
	\begin{align*}
		\CIntBody{C}{p} K \subseteq \CIntBody{C}{p} L \quad \quad \text{but} \quad \quad  -p\Vol{2n}{K}>-p\Vol{2n}{L}.
	\end{align*}
\end{proposition}
Let us point out that $\convexbodies(\IntBody_{C,p})^{\unitsurf{1}}$ consists precisely of all $\unitsurf{1}$-invariant convex bodies with non-empty interior (in particular, does not depend on $C$).

For the proof of Proposition~\ref{prop:counterexIDp}, we need the following perturbation statement, which is a slight generalization of \cite{Koldobsky2005}*{Lem.~4.10}. It follows from the fact that the condition in Proposition~\ref{prop:embedLp} is stable under suitable approximations.
\begin{lemma}\label{lem:embedApprox}
	Let $K \in \convexbodies(\RR^n)$ be origin-symmetric with $0 \in \interior K$, and suppose that $K$ does not embed into $L_p$. Then there exists a sequence of smooth origin-symmetric convex bodies $(K_k)_{k\in \NN} $ with strictly positive curvature such that $\rho_{K_k}^{-p} \to \rho_K^{-p}$ and each $K_k$ does not embed into $L_p$.
\end{lemma}

\begin{proof}[Proof of Proposition~\ref{prop:counterexIDp}]
Let $\widetilde{L} \in \convexbodies(\IntBody_{C,p})^{\unitsurf{1}}$ not in $\cl_w\left(\imgSetPhi{\CIntBody{C}{p}}{\starbodiesO}\right)=\{\hookrightarrow L_p\}^{\unitsurf{1}}$. By Lemma~\ref{lem:embedApprox}, there exists a smooth $L \in \convexbodies(\CC^n)$ with strictly positive curvature which does not embed into $L_p$. Moreover, we can assume that $L$ is $\unitsurf{1}$-invariant and, thus, $L \in \convexbodies(\IntBody_{C,p})$.

As $\rho_L$ is smooth on $\unitsurf{2n-1}$, so is $\sphericalFourier{p}\rho_L^{-p}$. By Proposition~\ref{prop:embedLp}, $-p\sphericalFourier{p}\rho_{L}^{-p}$ is negative on some $\unitsurf{1}$-invariant open set $U \subseteq\unitsurf{2n-1}$. Let $0 \neq \varphi \in C^\infty(\unitsurf{2n-1})$ be non-negative and $\unitsurf{1}$-invariant, such that $\varphi$ is supported inside $U$, and set $\psi = \sphericalFourier{p} \varphi$. Since $\sphericalFourier{p}$ intertwines rotations, $\psi$ is also $\unitsurf{1}$-invariant.

Next, we define $K \in \convexbodies(\CC^n)$ by
\begin{align*}
	\rho_{K}^{2n+p} = \rho_L^{2n+p} - \varepsilon\psi,
\end{align*}
for $\varepsilon > 0$ sufficiently small, using Lemma~\ref{lem:convradial}. Note that, by construction, $K$ is again $\unitsurf{1}$-invariant and therefore in $\convexbodies(\IntBody_{C, p})$. Since $(2\pi)^{2n}\varphi=\sphericalFourier{-2n-p}\psi$, we get that $\sphericalFourier{-2n-p} \psi$ is positive and, thus, by Theorem~\ref{mthm:JOp} and using that $-p \nu_{C,p} \geq 0$,
\begin{align*}
	-p(2\pi)^2 \JOp{C}{p}\rho_{K}^{2n+p} &= \frac{-p}{2n+p}\TT{\nu_{C,p}}\sphericalFourier{-2n-p}\left(\rho_L^{2n+p} - \varepsilon p\psi\right) \\
	&= -p(2\pi)^2\JOp{C}{p}\rho_{L}^{2n+p}  - \varepsilon \frac{-p}{2n+p}\nu_{C,p}(\unitsurf{1})\sphericalFourier{-2n-p}\psi \leq -p(2\pi)^2\JOp{C}{p}\rho_{L}^{2n+p}.
\end{align*}
Consequently, $\IntBody_{C, p} K \subseteq \IntBody_{C, p}L$. As $\varphi$ is positive on an open subset of $U$ and $-p\sphericalFourier{p}\rho_L^{-p}$ is negative on $U$, we obtain by \eqref{eq:FourierInverse} and the self-adjointness of $\sphericalFourier{p}$,
\begin{align*}
	-2np \DualMixedVol{-p}{K,L}&=-p\innerProduct{\rho_K^{2n+p}, \rho_L^{-p}} = - 2np \voln_{2n}(L) + p\varepsilon \innerProduct{(2\pi)^{-2n}\sphericalFourier{-2n-p}\psi, \sphericalFourier{p}\rho_L^{-p}}\\
	&=- 2np \voln_{2n}(L) - \varepsilon \innerProduct{\varphi, (-p)\sphericalFourier{p}\rho_L^{-p}} > - 2np \voln_{2n}(L).
\end{align*}
By the dual $L_p$-Minkowski inequality~\eqref{eq:dualLpMinkineq},
\begin{align*}
	- p \voln_{2n}(L) < -p \DualMixedVol{-p}{K,L} \leq -p \voln_{2n}(K)^{\frac{2n+p}{2n}}\voln_{2n}(L)^{-\frac{p}{2n}},
\end{align*}
that is, $-p \voln_{2n}(L) < -p \voln_{2n}(K)$, concluding the proof.
\end{proof}

\subsection{\nameref{sec:counterexHidim}}
By \Cref{prop:counterex} resp.\ \Cref{prop:counterexIDp}, it suffices to show that not every $\unitsurf{1}$-invariant convex body is contained in the closure of $\imgSetPhi{\CIntBody{C}{p}}{\starbodiesO}^{\unitsurf{1}}$ in order to obtain a negative answer to \Cref{prob:BPgeneralPhi}. Using examples from \cite{Koldobsky2008} of bodies that do not embed into some $L_p$, \Cref{prop:CIntEmbedLp} then completes the proof of \Cref{prop:counterExHidim} for $\IntBody_{C, p}$.
\begin{proof}[Proof of \Cref{prop:counterExHidim} for $\IntBody_{C, p}$]
By \cite[Thm.~4]{Koldobsky2008}, the unit ball of the complex space $\ell^q_n$, i.e.
	\begin{align*}
		B_q=\left\{(z_1, \dots, z_n) \in \CC^{n}:\sum_{j=1}^n |z_j|^{q} \leq 1\right\},
	\end{align*}
where $q>2$ does not embed into $L_p$ for $p \in (-2n+4,0)$. In particular, $B_q$ does not embed into $L_p$ for $p \in (-2,0)$ if $n\geq 3$. Moreover, by \cite{Koldobsky2005}*{Thm.~6.17}, $B^q$ cannot embed into $L_p$ for $p \in (0,2]$ if $n \geq 3$. By \cite[Thm.~6]{Koldobsky2008}, finally, there exist $2$-dimensional, $\unitsurf{1}$-invariant convex bodies that do not embed into $L_p$ for any $p>0$.

Hence, by \Cref{prop:CIntEmbedLp}, $\cl_w(\IntBody_{C,p}(\starbodiesO)) \cap \convexbodies(\IntBody_{C,p})^{\unitsurf{1}} \subsetneq \convexbodies(\IntBody_{C, p})^{\unitsurf{1}}$.
\end{proof}

\subsection{\nameref{sec:affAnsLoDim}} In the case $n=2$, $C=\DD$ and $p \in (-2,0)$, embedding results from \cite{Koldobsky2008} yield a positive answer. Note that we need to take similar steps as in the proof of Proposition~\ref{prop:AffirmImage} here, as we did not extend the notion of complex $L_p$-intersection body (and thereby the statement of Proposition~\ref{prop:AffirmImage}) to measures (that is, take the closure of the image) as was done in the real setting.
\begin{proof}[Proof of \Cref{MainThmB}]
As $\convexbodies(\IntBody_{\DD, p})$ consists precisely of all $\unitsurf{1}$-invariant convex bodies with non-empty interior, by \cite[Thm.~3]{Koldobsky2008}, every $K \in \convexbodies(\IntBody_{\DD, p})$ in $\CC^2$ embeds into $L_p$ for $(-4,0)$. In particular, by \Cref{prop:embedLp}, $\sphericalFourier{p} \rho_K^{-p} \geq 0$ for all $K \in \convexbodies(\IntBody_{\DD, p})$, as $\GAMMA{-\frac{p}{2}} > 0$ for $p < 0$.

Suppose now that $K, L \in \convexbodies(\IntBody_{\DD, p})$ satisfy $\IntBody_{\DD, p} K \subseteq \IntBody_{\DD, p} L$. As $K$ and $L$ are $\unitsurf{1}$-invariant, by Proposition~\ref{prop:repIcpPicByFourier},
\begin{align*}
	\frac{\nu_{\DD,p}(\unitsurf{1})}{(2\pi)^2(4+p)}\sphericalFourier{-4-p}\rho_{K}^{4+p} = \rho_{\IntBody_{\DD, p} K}^{-p} \leq \rho_{\IntBody_{\DD, p} L}^{-p} = \frac{\nu_{\DD,p}(\unitsurf{1})}{(2\pi)^2(4+p)}\sphericalFourier{-4-p}\rho_{L}^{4+p},
\end{align*}
that is, $\sphericalFourier{-4-p} \rho_{K}^{4+p} \leq \sphericalFourier{-4-p} \rho_{L}^{4+p}$, as $\nu_{\DD,p} \geq 0$ by Proposition~\ref{prop:FphCmeasure}. By \eqref{eq:FourierInverse} and the self-adjointness of the spherical Fourier transform, we thus obtain
\begin{align*}
	4 \voln_{4}(K) &= 4 \DualMixedVol{-p}{K,K} = \innerProduct{\rho_K^{4+p}, \rho_K^{-p}} = (2\pi)^4 \innerProduct{\sphericalFourier{-4-p}\rho_K^{4+p}, \sphericalFourier{p}\rho_K^{-p}}\\
	&\leq (2\pi)^4 \innerProduct{\sphericalFourier{-4-p}\rho_L^{4+p}, \sphericalFourier{p}\rho_K^{-p}} = \innerProduct{\rho_L^{4+p}, \rho_K^{-p}} = 4 \DualMixedVol{-p}{L,K}.
\end{align*}
The dual $L_p$-Minkowski inequality~\eqref{eq:dualLpMinkineq} therefore implies
\begin{align*}
	\voln_{4}(K) \leq \DualMixedVol{-p}{L,K} \leq \voln_{4}(L)^{\frac{4+p}{4}}\voln_{4}(K)^{-\frac{p}{4}},
\end{align*}
that is, $\voln_{4}(K) \leq \voln_{4}(L)$, concluding the proof.
\end{proof}

\medskip

We conclude the treatment of complex $L_p$-intersection bodies by a remark on the complex centroid body.

\begin{remark}\label{rem:compCentroidBody}
The \emph{complex centroid body} map $\Gamma_C$ was defined in \cite{Haberl2019} for all $K \in \convexbodies(\CC^n)$ with non-empty interior by
\begin{align*}
	h_{\Gamma_CK}(u)=\frac{1}{\Vol{2n}{K}}\int_{K}h_{Cu}(x) dx, \quad u \in \unitsurf{2n-1}.
\end{align*}
Clearly, by \eqref{eq:defLpCIntBody}, $(\Gamma_C K)^\circ = \voln_{2n}(K) \IntBody_{C, 1}K = \IntBody_{C, 1}(\voln_{2n}(K)^{-1/(2n+1)} K)$, and, thus, $\convexbodies(\Gamma_C^\circ) = \convexbodies(\IntBody_{C, 1})$ and the images of $\Gamma_C^\circ$ and $\IntBody_{C, 1}$ coincide (as $\IntBody_{C, 1}$ is homogeneous).

Moreover, since
\begin{align*}
	\Gamma_CK \subseteq \Gamma_C L \quad \Leftrightarrow\quad \Gamma_C^\circ L \subseteq \Gamma_C^\circ K,
\end{align*}
the conditions in Problem~\ref{prob:BPgeneralPhi} are equivalent, with $K$ and $L$ replaced. \Cref{prop:AffirmImage} therefore implies
\begin{align*}
	\Vol{2n}{\Vol{2n}{K}^{-\frac{1}{2n+1}}K}\leq\Vol{2n}{\Vol{2n}{L}^{-\frac{1}{2n+1}}L},
\end{align*}
that is, $\voln_{2n}(K) \leq \voln_{2n}(L)$, whenever $L \in \IntBody_{\overline{C}, 1}(\starbodiesO) \cap \convexbodies(\IntBody_{C, 1})$.

In a similar way, by \Cref{prop:counterex} for $\IntBody_{C, 1}$, we can find bodies $K$ and $L$ with $\Gamma_C K \subseteq \Gamma_C L$ but $\Vol{2n}{L}<\Vol{2n}{K}$, if there are bodies outside the image, as is the case for all $n \geq 2$ by the same statement for $\IntBody_{C, 1}$, solving \Cref{prob:BPgeneralPhi} for $\Gamma_C$ as well.
\end{remark}

\section{Busemann--Petty problem for complex projection bodies}\label{sec:projbodies}
In this section, we prove Theorem~\ref{mthm:stateComplBP} for the complex projection body map $\Pi_{C}$. As for $\IntBody_{C, p}$, we will follow the steps outlined in the introduction.

\subsection{\nameref{sec:detInjSet}}
In terms of the multipliers~\eqref{eq:multJOp} of the operator $\JOp{C}{1}$, the injectivity set $\convexbodies(\Pi_{C})$, defined in \eqref{eq:InjSetPiC}, is given by
\begin{align*}
	\convexbodies(\Pi_{C}) = \left\{K \in \convexbodies(\CC^n): \interior K \neq \emptyset \text{ and } c_{l-k}[h_C^p]=0 \Rightarrow \pi_{k,l}\surfArea{K}=0 \right\}.
\end{align*}
In particular, we have $\Pi_{C} K = \Pi_{C} L$ for $K, L \in \convexbodies(\Pi_C)$ only if $S_K = S_L$, that is, $K = L$ up to translations. Moreover, $\convexbodies(\Pi_{C})$ contains the set of $\unitsurf{1}$-invariant convex bodies with non-empty interior.

In the following, we will need the notion of the mixed volume of $K, L \in \convexbodies(\RR^m)$, 
	\begin{align*}
		\MixedVol{K,L}=\frac{1}{m}\int_{\unitsurf{m-1}}h_L(u) d\surfArea{K}(u),
	\end{align*}
see, e.g., \cite{Schneider2014}*{Sec.~5}. Clearly, $\MixedVol{K,\cdot}$ is monotone and $\MixedVol{K, K} = \voln_{m}(K)$. Minkow\-ski's first inequality (see \cite[Thm.~7.2.1]{Schneider2014}) states that
	\begin{align}\label{eq:Minkineq}
		\MixedVol{K,L}\geq \Vol{m}{K}^{\frac{m-1}{m}}\Vol{m}{L}^{\frac{1}{m}}
	\end{align}
for all $K, L \in \convexbodies(\RR^m)$ with non-empty interior. Equality holds if and only if $K$ and $L$ are homothetic, that is, $K = \lambda L + x$ for some $\lambda \in \RR\setminus\{0\}$ and $x \in \RR^m$.

\medskip

We can now prove \Cref{prop:injecset} using again a small perturbation.

\begin{proof}[Proof of \Cref{prop:injecset} for $\Pi_{C}^\circ$]
	Let $L \in \convexbodies(\CC^n)$ be smooth with positive curvature, that is, $S_L$ has a smooth, positive density $s_L \in C^\infty(\unitsurf{2n-1})$. As $\convexbodies(\Pi_{C})$ is a proper subset, there exists $(k,l) \in \NN$ such that $c_{l-k}[h_C] = 0$. Using the Jacobi polynomial $\widetilde{P}_{k,l} \in \mathcal{H}^{2n}_{k,l}$, we define
	\begin{align*}
		\varphi = s_L + \varepsilon \widetilde{P}_{k,l},
	\end{align*}
	where $|\varepsilon|$ is sufficiently small such that $\varphi > 0$ and the sign of $\varepsilon$ is chosen such that
	\begin{align}\label{eq:prfInjecSetPicSignEps}
		\varepsilon \innerProduct{\widetilde{P}_{k,l},h_L}\leq 0
	\end{align}
	By Minkowski's existence theorem \cite{Schneider2014}*{Thm.~8.2.2}, there exists a convex body $K \in \convexbodies(\CC^n)$ such that $S_K = \varphi$.
	
	By the choice of $(k,l)$, $\TT{S_{iC}} \widetilde{P}_{k,l} = 0$ by \eqref{eq:multTnuCp}, and we conclude by Proposition~\ref{prop:repIcpPicByFourier}
	\begin{align*}
		-4\pi h_{\Pi_C K} = T_{S_{iC}}\sphericalFourier{-2n-1}S_K = T_{S_{iC}}\sphericalFourier{-2n-1}S_L = -4\pi h_{\Pi_C L},
	\end{align*}
	that is, $\Pi_{C}K = \Pi_{C} L$. Consequently, by \eqref{eq:prfInjecSetPicSignEps},
	\begin{align*}
		2n \MixedVol{K,L} = \innerProduct{S_K, h_L} = 2n\voln_{2n}(L) + \varepsilon\innerProduct{\widetilde{P}_{k,l}, h_L} \leq 2n \voln_{2n}(L).
	\end{align*}
	Minkowski's first inequality~\eqref{eq:Minkineq} then implies
	\begin{align*}
		\voln_{2n}(L) \geq \MixedVol{K, L} \geq \voln_{2n}(K)^{\frac{2n-1}{2n}}\voln_{2n}(L)^{\frac{1}{2n}},
	\end{align*}
	that is, $\voln_{2n}(L) \geq \voln_{2n}(K)$. As $K$ and $L$ are not homothetic, the equality cases of \eqref{eq:Minkineq} imply that the inequality is strict, concluding the proof.
\end{proof}

\subsection{\nameref{sec:affirmCases}}
As in case of $\IntBody_{C, p}$, Problem~\ref{prob:BPgeneralPhi} for $\Pi_C$ has an affirmative answer, when one of the bodies is in the image of the adjoint operator $\Pi_{\overline{C}}$. The proof idea is exactly the same as before, namely using an adjointness property for mixed volumes and applying Minkowski's first inequality.
\begin{proof}[Proof of Proposition~\ref{prop:AffirmImage} for $\Pi_C$]
First, we apply Proposition~\ref{prop:repIcpPicByFourier} and use that $\sphericalFourier{-2n-1}$ is self-adjoint and, by \Cref{lem:Tmult}, $\TT{S_{iC}}$ is the adjoint operator of $\TT{S_{i\overline{C}}}$, to obtain
\begin{align}
	2n\MixedVol{K, \Pi_{\overline{C}} L} &= \innerProduct{S_K, h_{\Pi_{\overline{C}} L}} =  \frac{-1}{4\pi} \innerProduct{S_K, \TT{S_{i\overline{C}}}\sphericalFourier{-2n-1}S_L} \nonumber\\&=\frac{-1}{4\pi} \innerProduct{\TT{S_{iC}}\sphericalFourier{-2n-1}S_K, S_L}= \innerProduct{h_{\Pi_{C} K}, S_L} = 2n\MixedVol{L, \Pi_{C} K}, \label{eq:prfAffirmImagePicAdj}
\end{align}
for every $K, L \in \convexbodies(\Pi_{C})$. Let now $K, L \in \convexbodies(\Pi_C)$ with $K = \Pi_{\overline{C}}K_0$ for some $K_0\in \convexbodies(\Pi_{C})$ and assume that $\Pi_{C}^\circ K \subseteq \Pi_{C}^\circ L$, equivalently, $\Pi_{C} L \subseteq \Pi_{C} K$. Then, by \eqref{eq:prfAffirmImagePicAdj} and the monotonicity of mixed volume,
\begin{align*}
	\voln_{2n}(K) = \MixedVol{K, \Pi_{\overline{C}}K_0} = \MixedVol{K_0, \Pi_C K} \geq \MixedVol{K_0, \Pi_{C}L} = \MixedVol{L, \Pi_{\overline{C}} K_0} = \MixedVol{L, K}.
\end{align*}
Minkowski's first inequality~\eqref{eq:Minkineq} finally implies that
\begin{align*}
	\voln_{2n}(K) \geq \MixedVol{L, K} \geq \voln_{2n}(L)^{\frac{2n-1}{2n}}\voln_{2n}(K)^{\frac{1}{2n}},
\end{align*}
that is, $\voln_{2n}(K) \geq \voln_{2n}(L)$ as claimed. Equality $\voln_{2n}(K) = \voln_{2n}(L)$ holds here only if there is equality in Minkowski's first inequality, that is, $K$ and $L$ are homothetic. As they have the same volume, they must be translates of each other.
\end{proof}

\subsection{\nameref{sec:outsideIm}} 
The next step is to show that the answer to Problem~\ref{prob:BPgeneralPhi} is negative if there are bodies which are not contained in the image of $\Pi_C$. Again, we will start with a description of $\Pi_{C}(\convexbodies)$ in terms of the spherical Fourier transform.
\begin{proposition}\label{prop:ImPiCPolarEmbedL1}
	Let $C \in \convexbodies(\CC)$ be origin-symmetric with non-empty interior. If $K\in \Pi_C(\convexbodies)$, then 
	\begin{align*}
		\frac{1}{\GAMMA{-\frac 12}}\sphericalFourier{1}h_K\geq 0. 
	\end{align*}
	In particular, $K^\circ \in \{\hookrightarrow L_1\} \cap \{L \in \convexbodies(\CC^n): \rho_L^{-1} \in \image\TT{S_{iC}}\}$.
\end{proposition}

\begin{proof}
	First note that $\GAMMA{-\frac{1}{2}} = -2\sqrt{\pi}$. Combining \Cref{prop:repIcpPicByFourier} and \eqref{eq:FourierInverse}, using that multiplier transforms commute, we obtain
\begin{align*}
	-\sphericalFourier{1}h_K = \frac{1}{4\pi} \sphericalFourier{1}\TT{S_{iC}}\sphericalFourier{-2n-1}S_K = \frac{(2\pi)^{2n}}{4 \pi}\TT{S_{iC}}S_K \geq 0,
\end{align*}
yielding the first claim. The second claim follows directly from Proposition~\ref{prop:embedLp}.
\end{proof}
As before, it is sufficient to consider only $\unitsurf{1}$-invariant bodies. In the following, we will denote by $\Pi_C^\circ(\convexbodies) = \{ K^\circ: \, K \in \Pi_C(\convexbodies)\}$.
\begin{proposition}\label{prop:S1invPicsubimgDD}
	Let $C\in \convexbodies(\CC)$ be origin-symmetric with non-empty interior. Then
	\begin{align*}
		\Pi_C(\convexbodies)^{\unitsurf{1}}= \Pi_\DD(\convexbodies) \quad \text{and} \quad \cl_w\left(\Pi_\DD^\circ(\convexbodies)\right)=\{\hookrightarrow L_1\}^{\unitsurf{1}},
	\end{align*}
	where we denote by $\cl_w$ the closure with respect to weak convergence of $\rho_K^{-1}$.
\end{proposition}	
\begin{proof}
Let $K \in \Pi_C(\convexbodies)^{\unitsurf{1}}$, that is, $\unitsurf{1}$-invariant and such that $K = \Pi_C K_0$ for some $K_0 \in \convexbodies(\Pi_C)$. By the $\unitsurf{1}$-invariance, $\TT{S_{iC}}h_K = S_{iC}(\unitsurf{1}) h_K$ for all $C$ and, by Proposition~\ref{prop:repIcpPicByFourier},
\begin{align*}
	S_{i\DD}(\unitsurf{1}) h_K = \TT{S_{i\DD}} h_K = \frac{-1}{4\pi}\TT{S_{i\DD}}\sphericalFourier{-2n-1}\TT{S_{iC}}S_{K_0}. 
\end{align*}
Setting $\mu = S_{i\DD}(\unitsurf{1})^{-1}\TT{S_{iC}}S_{K_0}$, we note that $\mu$ is a centered (even), non-negative measure, which is not concentrated in a subsphere. Hence, by Minkowski's existence theorem~\cite{Schneider2014}*{Thm.~8.2.2}, there exists an origin-symmetric body $L_0 \in \convexbodies(\CC^n)$ with non-empty interior such that $S_{L_0} = \mu$.

Consequently, $K = \Pi_\DD L_0$, that is, $K \in \Pi_\DD(\convexbodies)$. Repeating the argument with $C$ and $\DD$ replaced yields the other inclusion.

Next, by \Cref{prop:ImPiCPolarEmbedL1}, all bodies in $\Pi_\DD^\circ(\convexbodies)$ embed into $L_1$ and are $\unitsurf{1}$-invariant. As $\{\hookrightarrow L_1\}^{\unitsurf{1}}$ is closed under weak convergence, it remains to show that $\{\hookrightarrow L_1\}^{\unitsurf{1}} \subseteq \cl_w\left(\Pi_\DD^\circ(\convexbodies)\right)$.

To this end, let $K \in \{\hookrightarrow L_1\}^{\unitsurf{1}}$. By Proposition~\ref{prop:embedLp}, $\mu = -S_{i\DD}(\unitsurf{1})^{-1}\sphericalFourier{1}\rho_{K}^{-1}$ is a non-negative measure. We can therefore find a sequence of positive, $\unitsurf{1}$-invariant functions $f_k \in C^\infty(\unitsurf{2n-1})$, $k \in \NN$, which converges to $\mu$ in the weak topology. As every $f_k$ is, in particular, even and positive, Minkowski's existence theorem implies that there are convex bodies $K_k \in \convexbodies(\CC^n)$ with non-empty interior, such that $S_{K_k} = f_k$, $k \in \NN$. Then, by $\unitsurf{1}$-invariance,
\begin{align*}
	\rho_{\Pi_\DD^\circ K_k}^{-1} = -\frac{S_{i\DD}(\unitsurf{1})}{4\pi} \sphericalFourier{-2n-1}S_{K_k} \quad\stackrel{w}{\longrightarrow}\quad -\frac{S_{i\DD}(\unitsurf{1})}{4\pi} \sphericalFourier{-2n-1}\mu = \frac{(2\pi)^{2n}}{4\pi} \rho_{K}^{-1},
\end{align*}
concluding the proof.
\end{proof}
Next, we will use the characterization in Proposition~\ref{prop:S1invPicsubimgDD} to prove Proposition~\ref{prop:counterex}. For the reader's convenience, we repeat the statement using the notation of Proposition~\ref{prop:S1invPicsubimgDD}.
\begin{proposition}\label{prop:counterexPiD}
	If $\cl_w(\Pi_C^\circ(\convexbodies)) \cap \convexbodies(\Pi_C)^{\unitsurf{1}}$ is a strict subset of $\convexbodies(\Pi_C)^{\unitsurf{1}}$, then there exist $\unitsurf{1}$-invariant $K, L \in \convexbodies(\Pi_C)$ such that
	\begin{align*}
		\Pi_C K \subseteq \Pi_C L \quad \quad \text{but} \quad \quad  \Vol{2n}{K}>\Vol{2n}{L}.
	\end{align*}
\end{proposition}
Let us point out that $\convexbodies(\Pi_C)^{\unitsurf{1}}$ consists precisely of all $\unitsurf{1}$-invariant convex bodies with non-empty interior (in particular, does not depend on $C$ and is invariant under taking polars).

\begin{proof}
Let $\widetilde{L}^\circ \in \convexbodies(\Pi_C)^{\unitsurf{1}}$ not in $\cl_w(\Pi_C^\circ(\convexbodies)) = \{\hookrightarrow L_1\}^{\unitsurf{1}}$. By Lemma~\ref{lem:embedApprox}, there exists a smooth $L^\circ \in \convexbodies(\CC^n)$ with strictly positive curvature which does not embed into $L_1$. Moreover, we can assume that $L^\circ$ is $\unitsurf{1}$-invariant and, thus, $L \in \convexbodies(\Pi_{C})$.

As $h_L$ is smooth on $\unitsurf{2n-1}$, so is $\sphericalFourier{1}h_L$. By Proposition~\ref{prop:embedLp} for $L^\circ$, $- \sphericalFourier{1}h_L$ is negative on some $\unitsurf{1}$-invariant open set $U \subset \unitsurf{2n-1}$. Let $0 \neq \varphi \in C^\infty(\unitsurf{2n-1})$ be non-negative and $\unitsurf{1}$-invariant, such that $\varphi$ is supported inside $U$, and set $\psi = \sphericalFourier{1}\varphi$. Since $\sphericalFourier{1}$ intertwines rotations, $\psi$ is $\unitsurf{1}$-invariant.

Next, let $s_L$ denote the smooth density of $S_L$ and define
\begin{align*}
	f = s_L - \varepsilon \psi
\end{align*}
for $\varepsilon > 0$ sufficiently small such that $f$ is still positive. Then $f$ is, in particular, even and positive and by Minkowski's existence theorem there exists $K \in \convexbodies(\Pi_C)^{\unitsurf{1}}$ such that $S_K = f du$. Since $\sphericalFourier{-2n-1}\psi = (2\pi)^{2n} \varphi \geq 0$, and by Proposition~\ref{prop:repIcpPicByFourier},
\begin{align*}
	-4\pi h_{\Pi_C K} = \TT{S_{iC}}\sphericalFourier{-2n-1}S_K = -4\pi h_{\Pi_C L} - \varepsilon \TT{S_{iC}}\sphericalFourier{-2n-1}\psi\\
	= -4\pi h_{\Pi_C L} - \varepsilon S_{iC}(\unitsurf{1}) (2\pi)^{2n} \varphi \leq -4\pi h_{\Pi_C L},
\end{align*}
that is, $\Pi_C K \supseteq \Pi_C L$. As $\varphi$ is positive on an open subset of $U$ and $-\sphericalFourier{1}h_L$ is negative on $U$, we obtain by \eqref{eq:FourierInverse} and the self-adjointness of $\sphericalFourier{1}$
\begin{align*}
	2n \MixedVol{K, L} &= \innerProduct{S_K, h_L} = 2n\voln_{2n}(L) - \varepsilon\innerProduct{(2\pi)^{-2n}\sphericalFourier{-2n-1}\psi, \sphericalFourier{1}h_L} \\
	&= 2n\voln_{2n}(L) + \varepsilon\innerProduct{\varphi, -\sphericalFourier{1}h_L} < 2n\voln_{2n}(L).
\end{align*}
By Minkowski's first inequality~\eqref{eq:Minkineq},
\begin{align*}
	\voln_{2n}(L) > \MixedVol{K, L} \geq \voln_{2n}(K)^{\frac{2n-1}{2n}}\voln_{2n}(L)^{\frac{1}{2n}},
\end{align*}
that is, $\voln_{2n}(L)>\voln_{2n}(K)$, concluding the proof.
\end{proof}

\subsection{\nameref{sec:counterexHidim}}
As in the case of $\IntBody_{C, p}$, by Proposition~\ref{prop:counterex} resp.\ Proposition~\ref{prop:counterexPiD}, it suffices to show that not every $\unitsurf{1}$-invariant convex body is contained in the closure of $\Pi_C^\circ(\convexbodies)$ in order to solve Problem~\ref{prob:BPgeneralPhi}. Again, we will use examples from \cite{Koldobsky2008,Koldobsky2005} of bodies not embedding into $L_1$.
\begin{proof}[Proof of \Cref{prop:counterExHidim} for $\Pi_{C}$] In the proof of \Cref{prop:counterExHidim} for $\IntBody_{C, p}$ it was deduced from \cite{Koldobsky2008, Koldobsky2005} that the unit ball of the complex space $\ell^q_n$, $q>2$, does not embed into $L_1$ for $n \geq 2$. Hence the claim follows by \Cref{prop:counterexPiD}
\end{proof}

\bibliographystyle{abbrv}

\bibliography{ComplLpIntersBodies}

\end{document}